\documentclass{article}
\pdfoutput=1
\usepackage[utf8]{inputenc} 
\usepackage[T1]{fontenc}    
\usepackage{hyperref}       
\usepackage{url}            
\usepackage{booktabs}       
\usepackage{amsfonts}       
\usepackage{nicefrac}       
\usepackage{microtype}      

\usepackage[english]{babel}
\usepackage{amsmath}
\usepackage{mathrsfs}
\usepackage{amssymb}
\usepackage{graphicx}
\usepackage[colorinlistoftodos]{todonotes}
\usepackage{url}
\usepackage{hyperref}
\usepackage{amsthm}
\usepackage{csquotes}
\usepackage{comment}
\usepackage[ruled]{algorithm}
\usepackage{algorithmicx}
\usepackage[noend]{algpseudocode}
\algnewcommand\algorithmicinput{\textbf{def}}
\algnewcommand\def{\item[\algorithmicinput]}
\usepackage{enumitem}
\usepackage{wrapfig}
\usepackage{authblk}
\usepackage{lipsum}
\usepackage{xcolor}
\usepackage{dsfont}
\usepackage{import}
\usepackage{xifthen}
\usepackage{pdfpages}
\usepackage{transparent}

\DeclareMathOperator{\fp}{fp}
\DeclareMathOperator{\hyp}{hyp}
\DeclareMathOperator{\slim}{slim}
\DeclareMathOperator{\thin}{thin}
\DeclareMathOperator{\insize}{insize}
\DeclareMathOperator{\minsize}{minsize}
\DeclareMathOperator{\diam}{diam}
\DeclareMathOperator{\RRG}{RRG}
\DeclareMathOperator{\ER}{ER}
\DeclareMathOperator{\Var}{Var}
\DeclareMathOperator{\Bin}{Bin}

\usepackage[numbers]{natbib}
\usepackage{bibentry}

\theoremstyle{plain}

\newtheorem{thm}{Theorem}[section] 
\newtheorem*{thm-non}{Theorem} 

\theoremstyle{definition}
\newtheorem{defn}[thm]{Definition} 
\newtheorem{exmp}[thm]{Example} 
\newtheorem{lemm}[thm]{Lemma}

\newtheorem{coro}[thm]{Corollary}
\newtheorem{remark}[thm]{Remark}

\newenvironment{claim}[1]{\par\noindent\textbf{Claim:}\space#1}{}
\newenvironment{claimproof}[1]{\par\noindent\textbf{Proof:}\space#1}{\leavevmode\unskip\penalty9999 \hbox{}\nobreak\hfill\quad\hbox{$\qed$}}

\begin{document}

\title{Hyperbolicity, slimness, and minsize, on average}
\author[1]{Anna C. Gilbert\thanks{Email: anna.gilbert@yale.edu}, Joon-Hyeok Yim\thanks{Email: joon-hyeok.yim@yale.edu}}
\date{\today}
\affil[1]{\small{Yale University, New Haven, CT}}
\maketitle

\begin{abstract}
A metric space $(X,d)$ is said to be $\delta$-hyperbolic if $d(x,y)+d(z,w)$ is at most $\max(d(x,z)+d(y,w), d(x,w)+d(y,z))$ by $2 \delta$. A geodesic space is $\delta$-slim if every geodesic triangle $\Delta(x,y,z)$ is $\delta$-slim. It is well-established that the notions of $\delta$-slimness, $\delta$-hyperbolicity, $\delta$-thinness and similar concepts are equivalent up to a constant factor. In this paper, we investigate these properties under an average-case framework and reveal a surprising discrepancy: while $\mathbb{E}\delta$-slimness implies $\mathbb{E}\delta$-hyperbolicity, the converse does not hold. Furthermore, similar asymmetries emerge for other definitions when comparing average-case and worst-case formulations of hyperbolicity. We exploit these differences to analyze the random Gaussian distribution in Euclidean space, random $d$-regular graph, and the random Erd\H{o}s-R\'enyi graph model, illustrating the implications of these average-case deviations.
\end{abstract}

\section{Introduction}

There are two recent lines of work that we unite in this paper. One of them is in the realm of probability and statistical mechanics. The second is in the realm of more practical data analysis or geometric machine learning. First, Chatterjee and Sloman~\cite{chatterjee2021average} introduce the notion of average or expected Gromov hyperbolicity in the context of the Parisi ansatz and Ising models. Their main result is that if a metric space has small hyperbolicity on average (i.e., it is tree-like on average), then it can be approximately embedded in a tree with low distortion.

Simultaneously, there has been a resurgence in computing Gromov hyperbolicity values with the burgeoning interest in geometric graph neural networks. There are many proposals to improve the efficacy of graph neural nets by discovering the appropriate geometry and then embedding a data set comprised of a graph and associated node and/or edge features in that geometry. A natural geometry to use is hyperbolic space; thus, the necessity to determine just how hyperbolic a given graph data set is. 

The desire for methods to determine the ``geometry of a data set'' necessitate appropriate hyperbolicity measures that are adapted to finite and discrete structures. Furthermore, measures of hyperbolicity that are ``worst case'' are often too pessimistic for finite discrete data. As a result, Gilbert and Yim~\cite{yim2024fitting} adapted Chatterjee and Sloman's notion of average hyperbolicity for finite data structures and developed an algorithm for fitting tree-like on average data approximately into trees with low average distortion.

While Gromov hyperbolicity is just one measure of hyperbolicity for metric spaces, there are a number of different definitions that capture different geometric properties: slimness, thinness, and size of inscribed circles. These definitions are all equivalent up to various constants. In this paper, we adapt these geometric definitions to the finite discrete setting and define their average values. We do so in order to develop a robust set of potential algorithmic tools for analyzing the geometric properties of finite discrete structures. Surprisingly, we find that the equivalences in the \emph{usual worst case} definition do not necessarily hold in the \emph{average} definition and the constants are also not the same! 

Another key task in adapting and using these geometric tools is interpreting their results on finite data sets. In particular, it is crucial to understand how a data set might masquerade as almost hyperbolic. For example, do points drawn from a high dimensional Gaussian distribution exhibit low average hyperbolicity (despite residing in Euclidean space)? In addition, we to use these tools on graph data sets and so it is imperative to understand how they perform on families of random graphs. 

\begin{figure}[ht]
    \centering
    \def\svgwidth{0.9\columnwidth}
    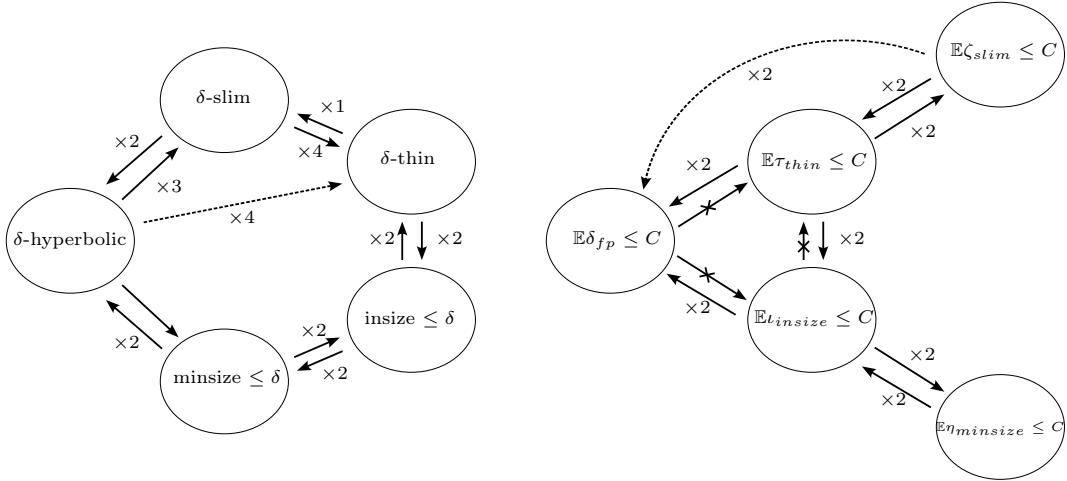
    \caption{Overview of the results: various $\delta$-hyperbolicity conditions are no longer equivalent when it comes to average. The figure highlights what are the (non-)equivalences.}
    \label{fig:delta-diagram}
\end{figure}

\begin{table}
\centering
\resizebox{\columnwidth}{!}{
\begin{tabular}{c|c|ccc}
    \textbf{Models} & (Usual) hyperbolicity & Average $\delta_{fp}$ & Average $\zeta_{slim}$ & Average $\eta_{minsize}$ \\
    \midrule
    $G_n$ (Section 3) & $\Theta(\sqrt{n})$ & $o(1)$ & $\Theta(\sqrt{n})$ & $\Theta(\sqrt{n})$ \\
    $H_{M,n}$ (Section 3) & $\Theta(M)$ & $o(1)$ & $\nearrow M$ & $o(1)$ \\
    $\RRG(n,d) (d \geq 3)$ (Section 4) & $\Theta(\log n)$ & $o(\omega(n))$ & $\Theta(\log n)$ & $\Theta(\log n)$ \\
    $\ER(n, \frac{\lambda}{n}) (\lambda \geq 4.67)$ (Section 5) & $\Theta(\log n)$ & $o(\log n)$ & $\Theta(\log n)$ & $\Theta(\log n)$ \\
    \bottomrule
\end{tabular}
}
\caption{The hyperbolicity constants on some graph models, and how it \emph{breaks} when it comes to the average. Note that $G_n$ has size $O(n^{2.5})$ and $H_{M,n}$ has size $O(n^2)$.}
\end{table}

\section{Preliminaries}
\label{sec:prelims}

We begin with the usual definitions of Gromov hyperbolicity. For all of our definitions, there is an important distinction between a geodesic metric space and a discrete, combinatorial one. For graphs, this distinction is captured by treating edges in a graph as a curve in the metric space (upon which we can place intermediate points) or as a discrete object which is not a part of the metric space and, hence, cannot be further subdivided by intermediary points. This distinction causes some slight differences in definitions which we articulate as well.

\subsection{(Usual) definitions of hyperbolicity}

\begin{defn}[Gromov Hyperbolicity] Given a metric space $(X,d)$, the four-point condition of the quadruple $x,y,z,w$, $\operatorname{fp}$ is defined as
    \[ \operatorname{fp}(x,y,z,w) := \frac{1}{2} \left[ \max(P,Q,R) - \operatorname{med}(P,Q,R) \right], \]
    while
    \[ P := d(x,y)+d(z,w), Q:= d(x,z)+d(y,w), R:= d(x,w)+d(y,z). \]
    \begin{enumerate}[label=(\alph*)]
        \item Given a metric space $(X,d)$ and $\delta \geq 0$, if $\operatorname{fp}(x,y,z,w) \leq \delta$ holds for all $x,y,z,w \in X$, then we call $(X,d)$ \emph{$\delta$-hyperbolic}.
        \item Given a (connected) graph $G = (V,E)$ and $\delta \geq 0$, if $\operatorname{fp}(x,y,z,w) \leq \delta$ holds for all $x,y,z,w \in V$, then we will call $G$ as \emph{$\delta$-hyperbolic}. Here, we use shortest-path metric as the metric $d$.
    \end{enumerate}
\end{defn}

Note that the above hyperbolicity condition can also be defined using the Gromov product. The \emph{Gromov product} of two points $x,y \in X$ with respect to a base point $w \in X$ is defined as 
\[ 
    \langle x,y \rangle_w := \frac{1}{2}[d(x,w)+d(y,w)-d(x,y)].
\]
Then, the four point condition can be expressed as
\[
    \operatorname{fp}(x,y,z,w) = \max_{\pi \text{ perm}}[\min( \langle x,z \rangle_w, \langle y,z \rangle_w) - \langle x,y \rangle_w].
\]
The restriction to a base point $w$ in this definition does not cause a structural change in the definition of Gromov hyperbolicity; we do not discuss here the necessary but trivial modification. 

The more intuitive definition of hyperbolicity is that based upon the condition on a \emph{geodesic triangle}. Given three points $x,y,z$ in a \emph{geodesic space} $X$, we denote a \emph{geodesic triangle} $\Delta(x,y,z)$ with vertices $x,y,z$ as the union of all three geodesic segments $[x,y], [y,z], [z,x]$ (where $[p,q]$ denotes a geodesic segment with endpoints $p$ and $q$). 

\begin{defn}[$\delta$-slim] We start with, arguably the most familiar, the \emph{$\delta$-slim} definition.
    \begin{enumerate}[label=(\alph*)]
        \item Given a geodesic metric space $(X, d)$ and $\delta \geq 0$, a geodesic triangle $\Delta(x,y,z)$ for $x,y,z \in X$ is $\delta$-slim if for any point $w \in [x,y]$, the distance from $w$ to $[y,z] \cup [z,x]$ is at most $\delta$, and similarly for $[y,z]$ and $[z,x]$. We call $(X,d)$ \emph{$\delta$-slim} if every geodesic triangle in $X$ is $\delta$-slim.
        \item Given a (connected) graph $G = (V,E)$ and $\delta \geq 0$, a geodesic triangle $\Delta(x,y,z)$ for $x,y,z \in V$ is $\delta$-slim if for any \textcolor{red}{vertex} $w \in [x,y]$, the distance from $w$ to $[y,z] \cup [z,x]$ is at most $\delta$, and similarly for $[y,z]$ and $[z,x]$. We refer to $G$ as \emph{$\delta$-slim} if every geodesic triangle in $G$ is $\delta$-slim.
    \end{enumerate}
\end{defn}
Note for a graph there is a slight difference on the hyperbolicity constant. For example, a cycle with length 3 $C_3$ is 0-slim if we realize the space as a graph, as a geodesic segment only contains two endpoints. However, it is not 0-slim if we consider the space as a geodesic metric space.


A similar but more restricted condition is the following \emph{$\delta$-thin} condition.
\begin{defn}[$\delta$-thin]
    \begin{enumerate}[label=(\alph*)]
        \item Given a geodesic metric space $(X,d)$ and $\delta \geq 0$, a geodesic triangle $\Delta(x,y,z)$ for $x,y,z \in X$ is $\delta$-thin if $d(x,v) = d(x,w) \leq \langle y,z \rangle_x$, then $d(v,w) \leq \delta$ holds for any points $v \in [x,y]$ and $w \in [x,z]$, and similar conditions hold for $y$ and $z$. We say $(X,d)$ is \emph{$\delta$-thin} if every geodesic triangle in $X$ is $\delta$-thin.
        \item Given a (connected) graph $G = (V,E)$ and $\delta \geq 0$,  a geodesic triangle $\Delta(x,y,z)$ for $x,y,z \in X$ is $\delta$-thin if $d(x,v) = d(x,w) \leq \langle y,z \rangle_x$, then $d(v,w) \leq \delta$ holds for any \textcolor{red}{vertices} $v \in [x,y]$ and $w \in [x,z]$, and similar conditions hold for $y$ and $z$. We say $G$ is \emph{$\delta$-thin} if every geodesic triangle in $X$ is $\delta$-thin.
    \end{enumerate}
\end{defn}

Let us denote $\slim(\Delta(x,y,z))$ and $\thin(\Delta(x,y,z))$ as the infimum of the $\delta$-slim, thin, respectively, values for all triangles; i.e.,
\begin{align*} 
   \slim(\Delta(x,y,z)) &:= \inf \{ \delta \geq 0 : \Delta(x,y,z) \text{ is } \delta\text{-slim} \} \\
   \thin(\Delta(x,y,z)) &:= \inf \{ \delta \geq 0 : \Delta(x,y,z) \text{ is } \delta\text{-thin} \}.
\end{align*}

\begin{remark}
    By definition, a $\delta$-thin triangle is always $\delta$-slim. In other words, $\slim(\Delta(x,y,z)) \leq \thin(\Delta(x,y,z))$ always holds. Consequently, a $\delta$-thin space is always $\delta$-slim.
\end{remark}


\begin{figure}[ht]
    \centering
    \def\svgwidth{0.8\columnwidth}
    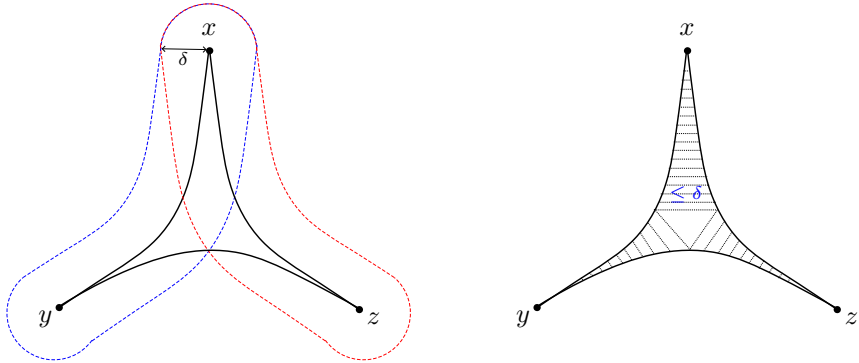
    \caption{$\delta$-slim and $\delta$-thin triangle}
    \label{fig:delta-slim-and-thin}
\end{figure}

\begin{defn}[minsize, insize]
    The \emph{minsize} of a geodesic triangle $\Delta(x,y,z)$ is defined as follows.
    \[ \minsize(\Delta(x,y,z)) = \inf_{x' \in [y,z], y' \in [z,x], z' \in [x,y]} \diam(\{x',y',z'\})\]
    In particular, pick $m_{yz} \in [y,z]$ with $d(x,y) + d(y,m_{yz}) = d(x,z) + d(z,m_{yz})$ and $m_{zx} \in [z,x], m_{xy} \in [x,y]$ similarly. We call the diameter of $\{m_{xy} , m_{yz}, m_{zx}\}$, the \emph{insize} of the triangle.
\end{defn}

\begin{remark}
    If given space is a graph $G = (V,E)$, by convexity, it turns out that $\inf_{x',y',z'} \operatorname{diam}(\{x',y',z'\})$ occurs when $x',y',z'$ are all vertices (not a point on an edge). Therefore, we do not need to specifically define the minsize of triangle on a graph. Note that even if $m_{xy}, m_{yz}, m_{zx}$ may not be vertices in a graph (instead, they are midpoint of an edge), we see that minsize $\leq$ insize.
\end{remark}



\subsection{Equivalences amongst definitions of (usual) hyperbolicity}
\label{subsec:usual_equivs}
The following results show that all of the above hyperbolicity definitions are equivalent up to multiplication by a small constant. In the next section, we see that this is no longer the case when we consider the \emph{average} definitions of hyperbolicity. As many of these results have appeared in the literature (and we are simply compiling them in one place for completeness) or are straightforward consequences of the definitions, we state the results here and for the proofs, either refer to the appropriate reference or show the result in the Appendix.

\begin{thm}
    Given a geodesic metric space $(X,d)$,
    \begin{enumerate}
        \item if it is $\delta$-hyperbolic, then it is $3\delta$-slim~\cite{alonso1991notes}.
        \item if it is $\delta$-thin, then it is $\delta$-hyperbolic~\cite{bridson2013metric}.
        \item if it is $\delta$-slim, then it is $4\delta$-thin~\cite{ghys1990espaces}. The bound is tight. 
        \item if it is $\delta$-slim, then it is $2\delta$-hyperbolic~\cite{soto2011quelques}. The bound is asymptotically tight.
        \item if every geodesic triangle has minsize at most $\delta$, then it has insize at most $3\delta$ as well~\cite{ghys1990espaces}. We note that the original reference uses $4\delta$, but we can easily improve the bound.
        \item if every geodesic triangle has insize at most $\delta$, then it is $2\delta$-thin~\cite{ghys1990espaces}.
    \end{enumerate}
\end{thm}

There are similar statements for discrete graphs. Proofs are in the Appendix.
\begin{thm}
    Given a graph $G = (V,E),$
    \begin{enumerate}
        \item if it is $\delta$-hyperbolic, then it is $3\delta + \frac{1}{2}$-slim.
        \item if it is $\delta$-thin, then it is $\delta + \frac{1}{2}$-hyperbolic.
        \item if it is $\delta$-slim, then it is $2\delta + \frac{1}{2}$-hyperbolic.
        \item if it is $\delta$-slim, then it is $4\delta$-thin.
    \end{enumerate}
\end{thm}

We conclude with the definitions of hyperbolicity of a geodesic metric space.
\begin{defn} Given a geodesic space or graph $X$, denote
    \[ \hyp(X):= \inf\{ \delta : X \text{ is } \delta \text{-hyperbolic}\} = \sup\{\fp(x,y,z,w)\}\]
    \[ \slim(X):= \inf\{ \delta : X \text{ is } \delta \text{-slim}\} = \sup\{\slim(\Delta(x,y,z))\}\]
    \[ \thin(X):= \inf\{ \delta : X \text{ is } \delta \text{-thin}\} = \sup\{\thin(\Delta(x,y,z))\}\]
    \[ \insize(X):= \sup\{ \text{insize}(\Delta) : \Delta \text{ geodesic triangle} \} = \sup\{\insize(\Delta(x,y,z))\}\]
    \[ \minsize(X):= \sup\{ \text{minsize}(\Delta) : \Delta \text{ geodesic triangle} \} = \sup\{\minsize(\Delta(x,y,z))\}\]
\end{defn}

\section{Average definitions of hyperbolicity}
\label{sec:average}

In this section, we explicate the \emph{average} version of each of the hyperbolicity definitions from the previous section. The idea underlying the average notions is that we select some number of points at random from a space $(X,d)$ or a graph $G=(V,E)$ and compute various geometric quantities. All of these definitions suppose that we select the points independently and identically distributed according to a distribution which we leave unspecified. We do, however, assume, that the distribution is non-trivial on all points in the space (or vertices of the graph).

\subsection{Definitions}

\begin{defn}
We define the \emph{average} hyperbolicity on $(X,d)$ as follows
\[ \mathbb{E}\delta_{hyp}(X) := \mathbb{E}_{x,y,z,w  \,\,\text{iid}} \fp(x,y,z,w). \]
\end{defn}
We define the slimness of a geodesic triangle and its average similarly. 

For graphs, we must be a bit more careful. It might be the case that the geodesic path between two vertices $x,y$ is not unique (i.e., there may be several shortest paths between two vertices), thereby making it difficult to define uniquely a geodesic triangle amongst three vertices $x,y,z$. In this case, we \emph{randomly} and \emph{uniformly} choose one segment amongst all shortest paths between $x,y$ to constitute a segment in $\Delta(x,y,z)$. To be more precise, suppose there are $\sigma_{x,y}$ shortest paths between $x$ and $y$. Then, we choose one path amongst them to form $[x,y] \in \Delta(x,y,z)$ with uniform probability $\frac{1}{\sigma_{x,y}}$, and similarly for $[y,z]$ and $[z,x]$.

\begin{defn}
We use the following notation to define the corresponding average hyperbolicity values on three vertices $x,y,z$:
\begin{enumerate}
    \item $\zeta(x,y,z) := \mathbb{E}\slim(\Delta(x,y,z))$
    \item $\tau(x,y,z) := \mathbb{E}\thin(\Delta(x,y,z))$
    \item $\iota(x,y,z) := \mathbb{E}\insize(\Delta(x,y,z))$
    \item $\eta(x,y,z) := \mathbb{E}\minsize(\Delta(x,y,z))$
\end{enumerate}
\end{defn}
For example, if the geodesic triangle on $x,y,z$ is unique, then $\zeta(x,y,z)$ is the slimness of the triangle itself.

\begin{defn}
We define the \emph{average} hyperbolicity values on $(X,d)$ as:
\begin{enumerate}
    \item $\mathbb{E}\zeta_{slim}(X) := \mathbb{E}_{x,y,z \text{iid}} \zeta(x,y,z)$
    \item $\mathbb{E}\tau_{thin}(X) := \mathbb{E}_{x,y,z \text{iid}} \tau(x,y,z)$
    \item $\mathbb{E}\iota_{insize}(X) := \mathbb{E}_{x,y,z \text{iid}} \iota(x,y,z)$
    \item $\mathbb{E}\eta_{minsize}(X) := \mathbb{E}_{x,y,z \text{iid}} \eta(x,y,z)$
\end{enumerate}
\end{defn}

\begin{remark}
    For geodetic graphs, the shortest path between two vertices $x$ and $y$ is always unique (See, for example, \cite{ore1962theory}). Therefore, $\zeta(x,y,z) = \slim(\Delta(x,y,z))$ always holds (and so on).
\end{remark}

\subsection{Equivalences: lower and upper bounds}
Given the equivalences amongst the usual definitions of hyperbolicity in Section~\ref{subsec:usual_equivs}, one might expect that the average definitions of hyperbolicity also enjoy similar equivalences. Surprisingly, this is not necessarily the case. There are some definitions of average hyperbolicity that are equivalent and there are some that are not. In this subsection, we detail the equivalences, exhibiting upper and lower bounds amongst the equivalent definitions. Because the non-equivalences and the examples that witness them are more compelling and suprising, we delay their presentation until Section~\ref{sec:non_equivs}

We begin with a simple lemma that relates the slimness or minsize of a geodesic triangle with its thinness or insize, respectively. 
\begin{lemm}\label{lemm:geo-tri-equiv}
Given a geodesic space $(X,d)$ or a graph $G = (V,E)$,
\begin{enumerate}
    \item If a geodesic triangle $\Delta(x,y,z)$ is $\delta$-slim, then it is $4\delta$-thin (both as a geodesic space, or as a graph).
    \item If a geodesic triangle $\Delta(x,y,z)$ has minsize $\delta$, then it has insize less than or equal to $3\delta$ (both as a geodesic space, or as a graph).
\end{enumerate}
\end{lemm}

\begin{proof}
We follow the proof of~\cite{ghys1990espaces} and extend it slightly to the case when the geodesic triangle is in a graph. Given a geodesic triangle $\Delta(x,y,z)$ with slimness $\delta$, suppose there exists $u \in [x,z], v \in [x,y]$ with $d(u,x) = d(v,x) \leq \langle y,z \rangle_x$ and $d(u,v) > 4 \delta$. Then,
\[ d(u,[x,y]) = \min (d(u,[x,v]), d(u,[v,y])) \geq \langle x,v \rangle_u, \langle v,y \rangle_u ,
\]
where $2 \langle x,v \rangle_u = d(x,u)+d(v,u)-d(x,v) = d(v,u) > 4 \delta$ and $2 \langle v,y \rangle_u = d(v,u) + d(y,u) - d(v,y) = d(v,u) + d(y,u) - (d(x,y) - d(x,v)) = d(v,u) + (d(y,u) + d(x,u) - d(x,y)) \geq d(v,u) > 4 \delta$. Therefore, $d(u,[x,y]) > 2 \delta$ holds.

Now, pick $p_y \in [x,u]$ with $d(p_y,u) = \delta$. Even in the discrete graph case, we can choose such a point since $\delta$ is integer. Then we have $d(p_y,[x,y]) \geq d(u,[x,y]) - d(u,p_y) > \delta$ and there should be $q \in [y,z]$ such that $d(p_y,q) \leq \delta$. On the other hand,
    \begin{align*}
        d(p_y,[y,z]) & \geq d(x,[y,z]) - d(x,p_y) \geq \langle y,z \rangle_x - d(x,p_y) \\
        & = d(m_y,x) - d(x,p_y) = d(p_y,m_y) = d(p_y,u) + d(u,m_y) = \delta + d(u,m_y).
    \end{align*} 
    Therefore, $u = m_y, v = m_z$ and $d(p_y,[y,z]) = \delta$. Also,
    \[ d(p_y,[y,z]) = \min(d(p_y,[m_x,z]), d(p_y,[y,m_x])) \geq \min(\langle m_x,z \rangle_{p_y}, \langle y,m_x \rangle_{p_y}),\]
    where $2 \langle m_x, z \rangle_{p_y} = d(m_x, p_y) + d(z, p_y) - d(m_x, z) = \delta + d(p_y, m_x)$ and $2 \langle y,m_x \rangle_{p_y} = d(y,p_y)+d(m_x, p_y) - d(y,m_x) \geq (d(x,y) - d(x,p_y)) + d(p_y, m_x) - d(m_z, y) = d(x,m_z) + d(p_y, m_x) - d(x,p_y) = \delta + d(p_y, m_x)$. Therefore,
    \[ 2 \delta \geq 2 d(p_y, [y,z]) \geq 2 \min(d(p_y,[m_x,z]), d(p_y,[y,m_x])) \geq \delta + d(p_y, m_x) \Rightarrow d(p_y, m_x) \leq \delta.\]
    Similarly, pick $p_z \in [x,m_z]$ with $d(p_z, m_z) = \delta$. Then we have $d(p_z, m_x) \leq \delta$. This shows
    \[ d(m_y, m_z) \leq d(m_y, p_y) + d(p_y, m_x) + d(m_x, p_z) + d(p_z, m_z) \leq 4 \delta,\]
    a contradiction. This completes the proof of part~1 of the Lemma. 

For the second part of the Lemma, we modify the proof from \cite{ghys1990espaces} slightly. Given a geodesic triangle $\Delta(x,y,z)$, pick $p_x, p_y, p_z$ in $[y,z],[z,x],[x,y]$ with diameter $\delta$ and let $m_x, m_y, m_z$ be the points of the inscribed triple. We will show that $d(m_x, m_y) \leq 3 \delta$. There are three cases to analyze:
    \begin{enumerate}
        \item $p_x \in [m_x, z]$ and $p_y \notin [m_y, z]$: Then we see that $d(p_y, m_y) + d(m_x, p_x) = d(p_y, z) - d(p_x, z) \leq d(p_x, p_y) \leq \delta$. Therefore,
        \[ d(m_x, m_y) \leq d(m_x, p_x) + d(p_x, p_y) + d(p_y, m_y) \leq \delta + \delta = 2 \delta .\]
        \item $p_x \notin [m_x, z]$ and $p_y \in [m_y, z]$: One can similarly show that $d(m_x, m_y) \leq 2 \delta$ as well.
        \item $p_x \in [m_x, z]$ and $p_y \in [m_y, z]$: We see that $p_z \in [x, m_z]$ or $p_z \in [m_z, y]$ holds. Without loss of generality, assume that the former is true. Then
        \[d(p_y, m_y) + d(m_z, p_z) = d(p_y, x) - d(x,p_z) \leq d(p_y, p_z) \leq \delta,\]
        and
        \[d(m_x, p_x) - d(m_z, p_z) = d(y, p_x) - d(y, p_z) \leq d(p_x, p_z) \leq \delta.\]
        Therefore, $d(m_x, p_x) + d(p_y, m_y) \leq 2 \delta$ and $d(m_x, m_y) \leq d(m_x, p_x) + d(p_x, p_y) + d(p_y, m_y) \leq 3 \delta$.
        \item $p_x \notin [m_x, z]$ and $p_y \notin [m_y, z]$: Again, without loss of generality, assume $p_z \in [x, m_z]$. Then
        \[d(p_x, m_x) + d(m_z, p_z) = d(p_z, y) - d(y,p_x) \leq d(p_x, p_z) \leq \delta,\]
        and
        \[d(m_y, p_y) - d(m_z, p_z) = d(x, p_z) - d(x, p_y) \leq d(p_y, p_z) \leq \delta.\]
        Therefore, $d(m_x, p_x) + d(p_y, m_y) \leq 2 \delta$ and $d(m_x, m_y) \leq d(m_x, p_x) + d(p_x, p_y) + d(p_y, m_y) \leq 3 \delta$.       
    \end{enumerate}
    The above calculations cover every case and show $d(m_x, m_y), d(m_y, m_z), d(m_z, m_x) \leq 3 \delta$ as desired for part~2 of the Lemma.
\end{proof}

The graph in Figure~\ref{fig:tight-triangles} shows that this inequality is tight.

\begin{figure}[ht]
    \centering
    \def\svgwidth{0.8\columnwidth}
\begingroup%
  \makeatletter%
  \providecommand\color[2][]{%
    \errmessage{(Inkscape) Color is used for the text in Inkscape, but the package 'color.sty' is not loaded}%
    \renewcommand\color[2][]{}%
  }%
  \providecommand\transparent[1]{%
    \errmessage{(Inkscape) Transparency is used (non-zero) for the text in Inkscape, but the package 'transparent.sty' is not loaded}%
    \renewcommand\transparent[1]{}%
  }%
  \providecommand\rotatebox[2]{#2}%
  \newcommand*\fsize{\dimexpr\f@size pt\relax}%
  \newcommand*\lineheight[1]{\fontsize{\fsize}{#1\fsize}\selectfont}%
  \ifx\svgwidth\undefined%
    \setlength{\unitlength}{881.82671014bp}%
    \ifx\svgscale\undefined%
      \relax%
    \else%
      \setlength{\unitlength}{\unitlength * \real{\svgscale}}%
    \fi%
  \else%
    \setlength{\unitlength}{\svgwidth}%
  \fi%
  \global\let\svgwidth\undefined%
  \global\let\svgscale\undefined%
  \makeatother%
  \begin{picture}(1,0.42790594)%
    \lineheight{1}%
    \setlength\tabcolsep{0pt}%
    \put(0,0){\includegraphics[width=\unitlength,page=1]{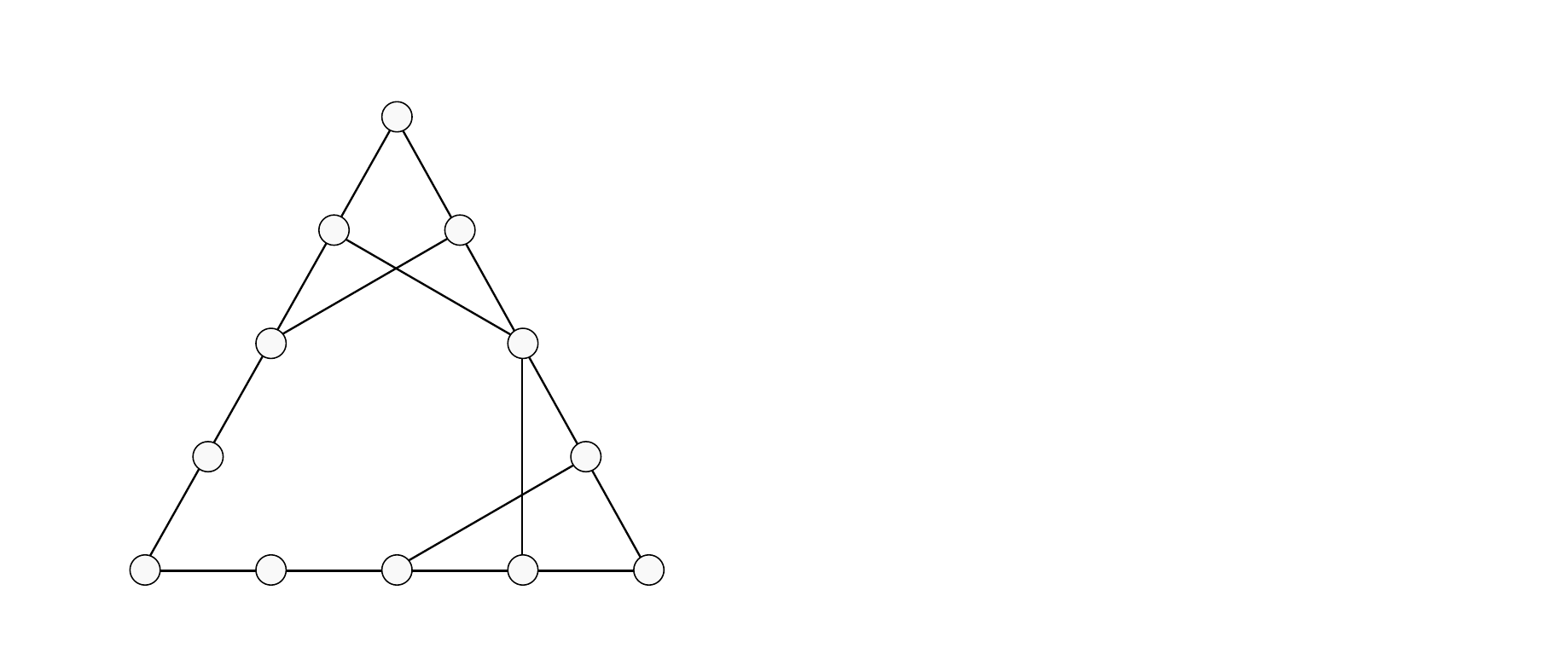}}%
    \put(0.24020672,0.37561041){\makebox(0,0)[lt]{\lineheight{1.25}\smash{\begin{tabular}[t]{l}$x_1$\end{tabular}}}}%
    \put(0.04899755,0.05232688){\makebox(0,0)[lt]{\lineheight{1.25}\smash{\begin{tabular}[t]{l}$y_1$\end{tabular}}}}%
    \put(0.43254733,0.05232688){\makebox(0,0)[lt]{\lineheight{1.25}\smash{\begin{tabular}[t]{l}$z_1$\end{tabular}}}}%
    \put(0,0){\includegraphics[width=\unitlength,page=2]{tight-triangles.pdf}}%
    \put(0.73946807,0.37570716){\makebox(0,0)[lt]{\lineheight{1.25}\smash{\begin{tabular}[t]{l}$x_2$\end{tabular}}}}%
    \put(0.54825887,0.05242361){\makebox(0,0)[lt]{\lineheight{1.25}\smash{\begin{tabular}[t]{l}$y_2$\end{tabular}}}}%
    \put(0.93180865,0.05242361){\makebox(0,0)[lt]{\lineheight{1.25}\smash{\begin{tabular}[t]{l}$z_2$\end{tabular}}}}%
    \put(0,0){\includegraphics[width=\unitlength,page=3]{tight-triangles.pdf}}%
  \end{picture}%
\endgroup%

    \caption{The left geodesic triangle is 1-slim, but 4-thin. The right geodesic triangle has minsize 1, but its insize is 3. Therefore we see that bounds on \ref{lemm:geo-tri-equiv} are tight and cannot be improved.}
    \label{fig:tight-triangles}
\end{figure}

\begin{coro}
    For a geodesic space or a graph $X$,
    \begin{enumerate}
        \item $\mathbb{E}\zeta_{slim}(X) \leq \mathbb{E}\tau_{thin}(X) \leq 4\mathbb{E}\zeta_{slim}(X)$.
        \item $\mathbb{E}\eta_{minsize}(X) \leq \mathbb{E}\iota_{insize}(X) \leq 3\mathbb{E}\eta_{minsize}(X)$.
        \item $\mathbb{E}\eta_{minsize}(X) \leq \mathbb{E}\iota_{insize}(X) \leq \mathbb{E}\tau_{thin}(X) (+1)$ (+1 is only for the case of a discrete graph).
    \end{enumerate}
\end{coro}

\subsection{Slim/thin triangles, on average, imply hyperbolicity, on average}
First, we begin with the proposition that the average four-point hyperbolicity condition is implied by the other (average) conditions, as we expect from the usual definitions. The main idea of the proof is to follow the proof details from \cite{soto2011quelques} and ~\cite{bridson2013metric} and simply take \emph{expectations}. For a discrete graph, however, we must make suitable, if straightforward modifications. In all cases, we assume that we draw points at random from a non-degenerate distribution.

\begin{thm} Given a geodesic space $(X,d)$,
    \begin{enumerate}
        \item $\mathbb{E}\delta_{hyp}(X) \leq 2 \mathbb{E}\zeta_{slim}(X)$ holds.
        \item $\mathbb{E}\delta_{hyp}(X) \leq \mathbb{E}\iota_{insize}(X) \leq \mathbb{E}\tau_{thin}(X)$ holds.
        \item $\mathbb{E}\delta_{hyp}(X) \leq 2 \mathbb{E}\eta_{minsize}(X) \leq 2 \mathbb{E}\iota_{insize}(X)$ holds.
    \end{enumerate}

Given a (connected) graph $G = (V,E)$,
    \begin{enumerate}
        \item $\mathbb{E}\delta_{hyp}(G) \leq 2 \mathbb{E}\zeta_{slim}(G) + \frac{1}{2}$ holds.
        \item $\mathbb{E}\delta_{hyp}(G) \leq \mathbb{E}\tau_{thin}(G) + \frac{1}{2}$ holds.
        \item $\mathbb{E}\delta_{hyp}(G) \leq 2 \mathbb{E}\eta_{minsize}(G) \leq 2 \mathbb{E}\iota_{insize}(G)$ holds.
    \end{enumerate}
\end{thm}

\begin{proof}
First, assume that $(X,d)$ is a geodesic space. The overall strategy is as follows. First, randomly sample $x,y,z$ and $w$. Next, we again \emph{randomly} sample geodesic paths between all pairs of those points, in order to constitute geodesic triangles $\Delta(x,y,z), \Delta(x,y,w), \Delta(x,z,w),$ and $\Delta(y,z,w)$. Then, we bound $\delta_{fp}(x,y,z,w)$ in terms of some constants (such as slimness, minsize, etc.) over these triangles. Finally, take the expectation to conclude the proof.
\begin{enumerate}
    \item We will show that
    \[ \delta_{fp}(x,y,z,w) \leq \frac{1}{2} \left[ \zeta(x,y,z) + \zeta(x,y,w) + \zeta(x,z,w) + \zeta(y,z,w) \right]. \]
    Without loss of generality, assume $d(x,y) + d(z,w) \geq \max(d(x,z) + d(y,w) , d(x,w) + d(y,z))$. 
    Pick $v \in [x,y]$ such that $d(x,z) - d(x,v) = d(y,z) - d(y,v) = \langle x,y \rangle_z$. Then we have $d(x,y) + d(z,t) \geq d(x,z) + d(y,t) = d(x,t) + d(y,z)$.
    Next, by assumption, there exists $t \in [x,w] \cup [y,w]$ such that $d(w,t) \leq \slim(\Delta(x,y,w))$. wlog assume $t \in [x,w]$. Then we have
    \begin{align*}
        d(x,y) + d(w,v) = d(x,v) + d(v,y) + d(w,v) & \leq (d(x,t) + d(t,v)) + d(y,v) + (d(t,v) + d(t,w)) \\
        & \leq d(x,w) + d(y,v) + 2 \slim(\Delta(x,y,w)).
    \end{align*}
    Similarly, one can show (because $d(y,v) + d(x,z) = d(x,v) + d(y,z)$,)
    \[ d(x,y) + d(z,v) \leq d(y,v) + d(x,z) + 2 \slim(\Delta(x,y,z)) = d(x,v) + d(y,z) + 2 \slim(\Delta(x,y,z)) .\]
    This shows that
    \begin{align*}
        d(x,y)+d(z,w) & \leq (d(x,y)+d(z,v)) + (d(x,y)+d(w,v)) - d(x,y) \\
        & \leq (d(x,v) + d(y,z) + 2 \slim(\Delta(x,y,z))) + (d(x,w) + d(y,v) + 2 \slim(\Delta(x,y,w))) - d(x,y) \\
        & = d(y,z) + d(x,w) + 2 (\slim(\Delta(x,y,z)) + \slim(\Delta(x,y,w))),
    \end{align*}
    which shows that $fp(x,y,z,w) \leq \slim(\Delta(x,y,z)) + \slim(\Delta(x,y,w))$. Taking the expectation yields $fp(x,y,z,w) \leq \zeta(x,y,z) + \zeta(x,y,w)$.
    
    We can repeat the argument above by suitably choosing taking $v \in [z,w]$. Then we have $fp(x,y,z,w) \leq \zeta(z,w,x) + \zeta(z,w,y)$ and leveraging the two inequalities proves the claim. The statement comes from taking the expectation (over the i.i.d.~samples $x,y,z,w$).
    \item Similarly, we will show
     \[ \delta_{fp}(x,y,z,w) \leq \frac{1}{4} \left[ \iota(x,y,z) + \iota(x,y,w) + \iota(x,z,w) + \iota(y,z,w) \right]. \]
     Without loss of generality, assume $d(x,z) + d(y,w) \geq d(x,w) + d(y,z) \geq d(x,y) + d(z,w)$. Pick $x' \in [x,w]$, $y' \in [y,w]$ such that $d(x',w) = d(y',w) = \langle x,y \rangle_w$ and $y'' \in [y,w]$, $z' \in [z,w]$ such that $d(y'',w) = d(z', w) = \langle y,z \rangle_w$. By assumption, we have $d(y', w) \geq d(y'', w)$. Then
     \begin{align*}
         d(x,z) & \leq d(x,x') + d(x',y') + d(y',y'') + d(y'',z') + d(z', z) \\
         & \leq d(x,x') + \insize(\Delta(x,y,w)) + d(y', y'') +  \insize(\Delta(y,z,w)) + d(z', z) \\
         & = \langle y,w \rangle_x + (\langle x,y \rangle_w - \langle y,z \rangle_w) + \langle y,w \rangle_z + \insize(\Delta(x,y,w)) + \insize(\Delta(y,z,w)) \\
         & = d(x,w) + d(y,z) - d(y,w) + \insize(\Delta(x,y,w)) + \insize(\Delta(y,z,w)),
     \end{align*}
     which immediately shows that $\delta_{fp}(x,y,z,w) \leq \frac{1}{2} \left[ \insize(\Delta(x,y,w)) + \insize(\Delta(y,z,w)) \right]$. Taking the expectation yields $\delta_{fp}(x,y,z,w) \leq \frac{1}{2} \left[ \iota(x,y,w) + \iota(y,z,w) \right]$.
     Again, we repeat the argument on $[x,z]$ and achieve $\delta_{fp}(x,y,z,w) \leq \frac{1}{2} \left[ \iota(x,y,z) + \iota(x,z,w) \right]$, which completes the proof.
     \item We will show
     \[ \delta_{fp}(x,y,z,w) \leq \frac{1}{2} \left[ \eta(x,y,z) + \eta(x,y,w) + \eta(x,z,w) + \eta(y,z,w) \right]. \]
     Again without loss of generality, assume $d(x,z) + d(y,w) \geq \max(d(x,y) + d(z,w) , d(x,w) + d(y,z))$. By assumption, pick $p_x \in [y,z], p_y \in [z,x], p_z \in [x,y]$ and $q_x \in [z,w], q_z \in [w,x], q_w \in [x,z]$ so that $\diam(\{p_x, p_y, p_z\}) = \minsize(\Delta(x,y,z))$ and $\diam(\{q_x, q_y, q_z\}) = \minsize(\Delta(x,z,w))$. WLOG assume $d(x, p_y) \leq d(x, q_w)$. Then
     \[d(y,w) \leq d(y, p_x) + d(p_x, p_y) + d(p_y, q_w) + d(q_w, q_z) + d(q_z, w)\]
     and
     \[d(x,z) + d(p_y, q_w) = d(x, q_w) + d(p_y, z) \leq d(x, q_z) + d(q_z, q_w) + d(p_y, p_x) + d(p_x, z),\]
     which shows
     \begin{align*}
        d(y,w) + d(x,z) & \leq d(y, p_x) + d(p_x, z) + d(x, q_z) + d(q_z, w) + 2 (d(p_x, p_y) + d(q_w, q_z)) \\
        & \leq d(y,z) + d(x,w) + 2 (\minsize(\Delta(x,y,z)) + \minsize(\Delta(x,z,w))).
     \end{align*}
     Hence, $\delta_{fp}(x,y,z,w) \leq \eta(x,y,z) + \eta(x,z,w)$ holds, and repeating the argument completes the proof.
\end{enumerate}
On a graph $G = (V,E)$, we need to modify the proof slightly as follows.
\begin{enumerate}
    \item It may be the case that we cannot explicitly pick $v \in [x,y]$ with the desired identity, if the corresponding Gromov product is a half-integer. Instead, we will assert $|(d(x,z) - d(x,v)) - (d(y,z) - d(y,v))| \leq 1$ and the constant 1/2 comes from the fact that two distance sum may differ by at most 1. Still, we are able to show $\delta_{fp}(x,y,z,w) \leq \slim(\Delta(x,y,z)) + \slim(\Delta(x,y,w)) + \frac{1}{2}$.

    \item (Note that for this proof we will use the \emph{thinness} condition instead.) Here, we will choose $d(x',w) = d(y',w) = \lfloor \langle x,y \rangle_w \rfloor$ and $d(y'',w) = d(z', w) = \lfloor \langle y,z \rangle_w \rfloor$. Then, by the thinness condition, we bound 
    $d(x',y') \leq \thin(\Delta(x,y,w))$ and $d(y'', z') \leq \thin(\Delta(y,z,w))$ and show 
    \[ \delta_{fp}(x,y,z,w) \leq \frac{1}{2} \left[ \insize(\Delta(x,y,w)) + \insize(\Delta(y,z,w)) + \frac{1}{2} \right].
    \]

    \item Finally, note that by the definition of minsize, we do not need to pick some additional nodes differently simply because we are in the discrete graph setting when it comes to the minsize. Therefore, $\mathbb{E}\delta_{hyp}(G) \leq 2 \mathbb{E}\eta_{minsize}(G) \leq 2 \mathbb{E}\iota_{insize}(G)$ still holds.
\end{enumerate}
\end{proof}

\begin{remark}
    The reader might wonder if we define $\zeta$ (and other constants) differently, how do our equivalences change? For example, we could define $\zeta(x,y,z)$ as the infimum (i.e., the best case) among all slimnesses over possible $\Delta(x,y,z)$. In this case, unfortunately, the desired equivalence inequalities do not hold. For example, if we consider the 4-cycle $C_4$, then we can always \emph{pick} the geodesic triangle so that its slimness is always zero. The graph, however, is not a tree (far from it) and it defies intuition that its hyperbolicity is zero. 
\end{remark}

\section{Non-equivalences}
\label{sec:non_equivs}

In this section, we detail the examples that demonstrate the non-equivalences of our definitions; that is, definition A does not imply defiition B up to multiplication by a (small) constant. . These results are suprising given the equivalences amongst all of the usual definitions of hyperbolicity. As there are a number of cases to examine, we break them up into subsections. In each subsection, we provide families of examples that illustrate the non-equivalences.

\subsection{Average hyperbolicity does not imply average slimness}

We answer, in the negative, whether hyperbolicity on average implies that triangles are slim/thin on average. We witness a family of graphs in which the answer to this question is no, average hyperboliticy does not necessarily imply thin or slim triangles on average. 


\begin{figure}[ht]
    \centering
    \def\svgwidth{0.33\columnwidth}
\begingroup%
  \makeatletter%
  \providecommand\color[2][]{%
    \errmessage{(Inkscape) Color is used for the text in Inkscape, but the package 'color.sty' is not loaded}%
    \renewcommand\color[2][]{}%
  }%
  \providecommand\transparent[1]{%
    \errmessage{(Inkscape) Transparency is used (non-zero) for the text in Inkscape, but the package 'transparent.sty' is not loaded}%
    \renewcommand\transparent[1]{}%
  }%
  \providecommand\rotatebox[2]{#2}%
  \newcommand*\fsize{\dimexpr\f@size pt\relax}%
  \newcommand*\lineheight[1]{\fontsize{\fsize}{#1\fsize}\selectfont}%
  \ifx\svgwidth\undefined%
    \setlength{\unitlength}{434.88488866bp}%
    \ifx\svgscale\undefined%
      \relax%
    \else%
      \setlength{\unitlength}{\unitlength * \real{\svgscale}}%
    \fi%
  \else%
    \setlength{\unitlength}{\svgwidth}%
  \fi%
  \global\let\svgwidth\undefined%
  \global\let\svgscale\undefined%
  \makeatother%
  \begin{picture}(1,0.78490094)%
    \lineheight{1}%
    \setlength\tabcolsep{0pt}%
    \put(0.10185091,0.61746526){\color[rgb]{0,0,0}\makebox(0,0)[lt]{\lineheight{1.25}\smash{\begin{tabular}[t]{l}$a_i$\end{tabular}}}}%
    \put(0.38526266,0.55875669){\color[rgb]{0,0,0}\makebox(0,0)[lt]{\lineheight{1.25}\smash{\begin{tabular}[t]{l}$\textcolor{blue}{b_{ijk}}$\end{tabular}}}}%
    \put(0.82964227,0.61978808){\color[rgb]{0,0,0}\makebox(0,0)[lt]{\lineheight{1.25}\smash{\begin{tabular}[t]{l}$a_j$\end{tabular}}}}%
    \put(0,0){\includegraphics[width=\unitlength,page=1]{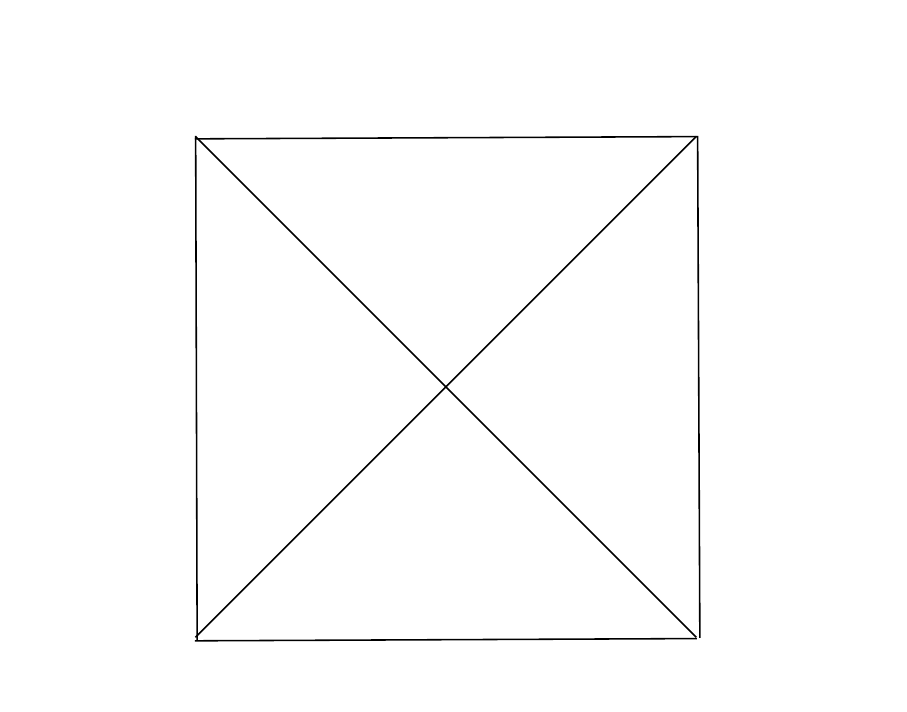}}%
    \put(0.24320809,0.69669717){\makebox(0,0)[lt]{\lineheight{1.25}\smash{\begin{tabular}[t]{l}{\scriptsize \textcolor{red}{$d(a_i, a_j) = 2 \lfloor \sqrt{n} \rfloor + 1$}}\end{tabular}}}}%
    \put(0,0){\includegraphics[width=\unitlength,page=2]{Gn.pdf}}%
  \end{picture}%
\endgroup%

    \caption{This figure shows the construction of the family of graphs $G_n$ in which average hyperbolicity does not necessarily imply thin or slim triangles on average.}
    \label{fig:extreme_gn}
\end{figure}

We construct a family of graphs $\{G_n\}$ for $n \geq 3$ as follows. Denote $m := 2 \lfloor \sqrt{n} \rfloor + 1$.
\begin{itemize}
    \item $V(G_n) = \{a_i | 1 \leq i \leq n\} \cup \{b_{ijk} | 1 \leq i < j \leq n \text{ and } 1 \leq k \leq m \}$.
    \item For any $1 \leq i < j \leq n$, connect $a_i - b_{ij1} - b_{ij2} - \cdots - b_{ijm} - a_j$.
\end{itemize}
Note that $|V(G_n)| = O(n^{2.5})$. We see that this graph family essentially has a \emph{gap} between the average hyperbolicity and the average slimness.
\begin{thm}
    The family of graphs $G_n$ satisfies
    \begin{enumerate}
        \item $\mathbb{E}\delta_{hyp}(G_n) \to 0$.
        \item $\mathbb{E}\zeta_{slim}(G_n), \mathbb{E}\eta_{minsize}(G_n) \to \infty$.
    \end{enumerate}
\end{thm}

\begin{proof}
    First, we show that $\mathbb{E}\delta_{hyp}(G_n) \to 0$. We randomly choose four points $x,y,z,w$ and consider the following \emph{generic} configuration:
    \begin{enumerate}
        \item None of the four points are the vertices $a_i$ for some $i$. This occurs with probability $1 - O(n^{-1.5})$.
        \item The point $x$ lies between $[a_{p(x)}, a_{q(x)}]$ for some $p(x), q(x) \in [n]$, and say $h_x :=d(x, a_{p(x)}) < d(x, a_{q(x)})$ (therefore, $p(x) < q(x)$ may not be true). Then with probability $1 - O(n^{-1})$, all of the 8 points $p(\cdot),q(\cdot)$ for each of $x,y,z,w$ are distinct.
    \end{enumerate}
    Overall, the \emph{generic} configuration occurs with probability $1 - O(n^{-1})$. In this case, the geodesic between two points is always unique and should pass $a_{p}$. Therefore,
    \[ d(x,y) = d(x, a_{p(x)}) + d(a_{p(x)}, a_{p(y)}) + d(a_{p(y)}, y) =m + h_x + h_y, \]
    and the similar formula holds for all other pairwise distances. Hence,
    \[ d(x,y) + d(z,w) = d(x,w)+d(y,z)=d(x,z)+d(y,w) = 2m + h_x + h_y + h_z + h_w,\]
    so that $\delta_{fp}(x,y,z,w) = 0$.

    On the other hand, if the point configuration is not \emph{generic}, then we bound $\delta_{fp}(x,y,z,w)$ by $\tfrac{1}{2}\diam(G_n)$, which is $4 \lfloor \sqrt{n} \rfloor+1 = 2m-1$. This implies that
    \[ \mathbb{E}\delta_{hyp}(G_n) \leq O(\tfrac{1}{n}) \cdot (2m-1) = O(\tfrac{1}{n} \cdot \sqrt{n}) = o(1).\]

\begin{figure}[h!]
    \centering
    \def\svgwidth{0.50\columnwidth}
    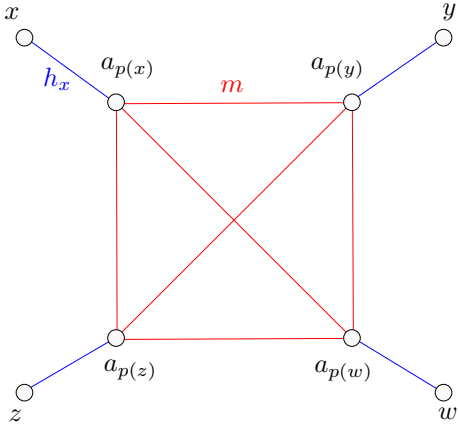
    \caption{This figure shows the generic configuration for the four point condition in the construction of the family $G_n$}
    \label{fig:generic-fp}
\end{figure}

    For the next proposition, it is enough to check that $\mathbb{E}\zeta_{slim}(G_n) \to \infty$. When we randomly sample three points $x,y,z$ in this setting, the \emph{generic} configuration is as follows:
    \begin{enumerate}
        \item None of the three points are the vertices $a_i$ for some $i$. This occurs with probability $1 - O(n^{-1.5})$.
        \item Again, the point $x$ lies between $[a_{p(x)}, a_{q(x)}]$ for some $p(x), q(x) \in [n]$, and say $d(x, a_{p(x)}) < d(x, a_{q(x)})$. Then with probability $1 - O(n^{-1})$, all 6 indices $p(\cdot), q(\cdot)$ for each of $x,y,z$ are distinct.
    \end{enumerate}
    Hence, the \emph{generic} configuration occurs with probability $1 - O(n^{-1})$. However, in this case, the geodesic triangle $\Delta(x,y,z)$ is unique and $\zeta(x,y,z) = \slim(\Delta(x,y,z)) = \lfloor \sqrt{n} \rfloor = \frac{m-1}{2}.$ Therefore,
    \[ \mathbb{E}\zeta_{slim}(G_n) \geq (1 - O(\tfrac{1}{n})) \cdot \frac{m-1}{2} = \Omega(\sqrt{n}).\]
\end{proof}


\subsection{Average slimness and average minsize are not necessarily equivalent}

In this subsection, we show that the slimness and minsize conditions are not necessarily equivalent in the average case. We have seen that slimness implies a bound on the minsize (and insize); this example indicates that the $\delta$-slim and $\delta$-thin conditions are the most strongest conditions amongst the equivalent hyperbolicity notions.

Fix a positive integer $M$. Then we construct a graph family $\{H_{M,n}\}$ for $n \geq 2$ as follows:
\begin{itemize}
    \item $V(H_{M,n})$ consists of what we call \emph{junction} nodes $\{c_1, \cdots, c_n\}$, nodes $\{a_{i,s}\}$ for $1 \leq i \leq n$, $1 \leq s \leq M \cdot n$ and nodes $\{b_{i,j,k}\}$ for $1 \leq i \leq n-1$, $1 \leq j \leq 2M$ and $k = 1,2$.
    \item For any $1 \leq i \leq n-1$, there are two paths between $c_i$ and $c_{i+1}$ with length $2M$: $c_i - b_{i,1,1} - b_{i,2,1} - \cdots - b_{i,2M,1} - c_{i+1}$ (informally call these upward paths) and $c_i - b_{i,1,2} - b_{i,2,2} - \cdots - b_{i,2M,2} - c_{i+1}$ (downward paths, respectively). There is an edge between $c_i$ and $a_{i,s}$ for any $i,s$.
\end{itemize}
Note that $|V(H_{M,n})| = O(n^{2})$. 

\begin{thm}
    Given any positive integer $M$, the graph family $H_{M,n}$ satisfies
    \begin{enumerate}
        \item $\mathbb{E}\zeta_{slim}(H_{M,n}) \to M$.
        \item $\mathbb{E}\eta_{minsize}(H_{M,n}) \to 0$.
    \end{enumerate}
\end{thm}


\begin{figure}[ht]
    \centering
    \def\svgwidth{0.8\columnwidth}
\begingroup%
  \makeatletter%
  \providecommand\color[2][]{%
    \errmessage{(Inkscape) Color is used for the text in Inkscape, but the package 'color.sty' is not loaded}%
    \renewcommand\color[2][]{}%
  }%
  \providecommand\transparent[1]{%
    \errmessage{(Inkscape) Transparency is used (non-zero) for the text in Inkscape, but the package 'transparent.sty' is not loaded}%
    \renewcommand\transparent[1]{}%
  }%
  \providecommand\rotatebox[2]{#2}%
  \newcommand*\fsize{\dimexpr\f@size pt\relax}%
  \newcommand*\lineheight[1]{\fontsize{\fsize}{#1\fsize}\selectfont}%
  \ifx\svgwidth\undefined%
    \setlength{\unitlength}{1391.84692383bp}%
    \ifx\svgscale\undefined%
      \relax%
    \else%
      \setlength{\unitlength}{\unitlength * \real{\svgscale}}%
    \fi%
  \else%
    \setlength{\unitlength}{\svgwidth}%
  \fi%
  \global\let\svgwidth\undefined%
  \global\let\svgscale\undefined%
  \makeatother%
  \begin{picture}(1,0.44245941)%
    \lineheight{1}%
    \setlength\tabcolsep{0pt}%
    \put(0,0){\includegraphics[width=\unitlength,page=1]{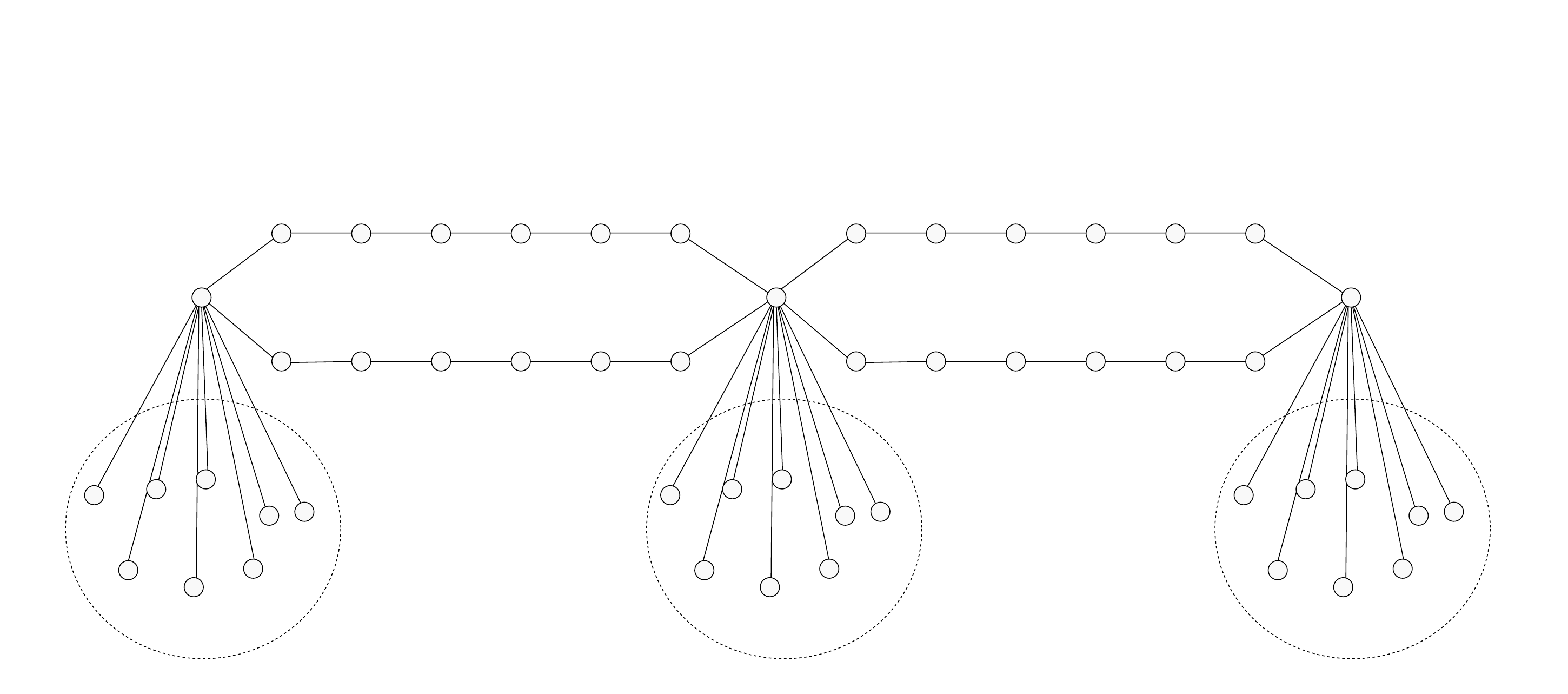}}%
    \put(0.48296543,0.26729716){\makebox(0,0)[lt]{\lineheight{1.25}\smash{\begin{tabular}[t]{l}$c_i$\end{tabular}}}}%
    \put(0.53784257,0.06713873){\makebox(0,0)[lt]{\lineheight{1.25}\smash{\begin{tabular}[t]{l}$a_{i,s}$\end{tabular}}}}%
    \put(0.5799981,0.31108773){\makebox(0,0)[lt]{\lineheight{1.25}\smash{\begin{tabular}[t]{l}$b_{i,k,1}$\end{tabular}}}}%
    \put(0.5799981,0.18632599){\makebox(0,0)[lt]{\lineheight{1.25}\smash{\begin{tabular}[t]{l}$b_{i,k,2}$\end{tabular}}}}%
    \put(0,0){\includegraphics[width=\unitlength,page=2]{Hmn.pdf}}%
  \end{picture}%
\endgroup%

    \caption{This figure depicts the graph family $H_{M,n}$ in which average slimnes and average minsize are not equivalent.}
    \label{fig:Hmn}
\end{figure}

\begin{proof}
    Our approach is to observe which configuration is generic and to argue that it occurs with high probability. Let us randomly sample $x,y,z$ in $V(H_{M,n})$ and consider the following generic configuration.
    \begin{enumerate}
        \item All points $x,y,z$ are in the ``cluster'' $\{a_{i,s}\}$, so that each is adjacent to $c_{i(x)}, c_{i(y)}, c_{i(z)}$, respectively.
        \item $|i(x) - i(y)|, |i(y) - i(z)|, |i(z) - i(x)| \geq \log_2 n$ holds.
    \end{enumerate}
    Note that first condition holds with probability at least $1 - O(n^{-1})$. Conditional on that, we see that the conditional distribution of $i(x)$s is uniform in $[n]$. Therefore, the second condition holds with (conditional) probability at least $1 - O(\log n/n)$. Thus, the \emph{generic} configratuion occurs with probability at least $1 - O(\log n/n)$.

    Now, consider a geodesic triangle $\Delta(x,y,z)$ in the generic configuration. Without loss of generality assume $i(x) < i(y) < i(z)$. Then we see that regardless of the choice of geodesic, $[x,y], [y,z]$ and $[z,x]$ always passes $c_{i(y)}$. In fact, $c_{i(y)}$ actually becomes the Steiner node among these points. Therefore, $\eta_{minsize}(x,y,z) = \minsize(\Delta(x,y,z)) = 0$ if $x,y,z$ are generic. On the other hand, it is easy to verify that every geodesic triangle in $H_{M,n}$ has minsize at most $2M$. Therefore,
    \[ \mathbb{E}\eta_{minsize}(H_{M,n}) \leq O(\tfrac{\log n}{n}) \cdot 2M = o(1),\]
    as $n \to \infty$.

    Now we evaluate the slimness of the generic triangle $\Delta(x,y,z)$. This depends on the choice of geodesic; let us consider $[x,y]$ and $[x,z]$ in $\Delta(x,y,z)$. When each geodesic runs from $c_{i(x)}$ to $c_{i(y)}$, if at least one of path choices between two adjacent junction nodes differs (i.e., one goes upward and the other goes downward), then the slimness of $\Delta(x,y,z)$ should be at least $M$, as the midpoint of each segment is distant from the other segment with distance $M$. As we have assumed three indices differ by at least $\log_2 n$, this phenomenon happens on at least $1 - n^{-1}$ portion of the geodesic triangles $\Delta(x,y,z)$ so that
    \[ \zeta_{slim}(x,y,z) \geq (1 - n^{-1}) \cdot M.\]
    Therefore,
    \[ \mathbb{E}\zeta_{slim}(H_{M,n}) \geq (1 - O(\tfrac{\log n}{n})) \cdot (1 - \tfrac{1}{n})M \to M,\]
    as $n \to \infty$. On the other hand, we can check that every geodesic triangle in $H_{M,n}$ has slimness at most $M$. Therefore, $\mathbb{E}\zeta_{slim}(H_{M,n}) \to M$ as desired.
\end{proof}

\subsection{Generic geodesic space examples}

In this subsection, we consider generic geodesic spaces, rather than families of graphs per se, and find more extreme examples.

\begin{exmp}
    Consider a Gaussian $N(0, \frac{I_d}{\sqrt{d}})$ random variable in $\mathbb{R}^d$. Note that Euclidean space is always non hyperbolic and we cannot bound the hyperbolicity constant $\delta$ in any sense in this space. However, if we sample points from a Gaussian distribution in $\mathbb{R}^d$ and compute the Gromov hyperbolicity constant for these points (embedded in $\mathbb{R}^d$, we see that as $d \to \infty$, $\delta_{fp} \to 0$. This is the distribution of distances between points $d(x,y)$ converges to $\sqrt{2}$. The geodesic triangle in $\mathbb{R}^d$, however, is \emph{nowhere} slim (or thin), so that $\zeta_{{\rm slim}}, \tau_{{\rm thin}}, \eta_{{\rm minsize}}, \iota_{{\rm insize}} \to \sqrt{2}/2$. 
\end{exmp}
This example highlights an important phenomenon. The first is that when one is computing hyperbolicity values numerically on finite data sets, it is important to be careful drawing conclusions. One's data set might be samples drawn from a Gaussian distribution in high dimensional Euclidean space and one might be tempted to conclude with a small Gromov hyperbolicity value, that the points are from a hyperbolic space when, in fact, they are not. 

\begin{exmp} 
    Another extreme example is a space equipped with a distribution concentrated on exactly 3 points uniformly, say $a,b,c \in X$. Then by definition, $\delta_{fp}(x,y,z,w) = 0$ always, as there would be a duplicate among four points. On the other hand, we see that $\mathbb{E}\zeta_{{\rm slim}}$ is $2/9 \cdot \zeta_{{\rm slim}}(a,b,c)$, and the similar thing holds on all other constants. This means that a single geodesic triangle can serve an example on the average case as well. For example, consider the graph $H$ such that $\Delta(x,y,z)$ has minsize(insize) 1 and slimness $n/2$. Equip this graph with a distribution concentrated on $x,y,z$ and we can see that \emph{the bound on minsize} does not imply \emph{the bound on slimness}.
\end{exmp}

\begin{figure}[ht]
    \centering
    \def\svgwidth{0.40\columnwidth}
\begingroup%
  \makeatletter%
  \providecommand\color[2][]{%
    \errmessage{(Inkscape) Color is used for the text in Inkscape, but the package 'color.sty' is not loaded}%
    \renewcommand\color[2][]{}%
  }%
  \providecommand\transparent[1]{%
    \errmessage{(Inkscape) Transparency is used (non-zero) for the text in Inkscape, but the package 'transparent.sty' is not loaded}%
    \renewcommand\transparent[1]{}%
  }%
  \providecommand\rotatebox[2]{#2}%
  \newcommand*\fsize{\dimexpr\f@size pt\relax}%
  \newcommand*\lineheight[1]{\fontsize{\fsize}{#1\fsize}\selectfont}%
  \ifx\svgwidth\undefined%
    \setlength{\unitlength}{393.69640921bp}%
    \ifx\svgscale\undefined%
      \relax%
    \else%
      \setlength{\unitlength}{\unitlength * \real{\svgscale}}%
    \fi%
  \else%
    \setlength{\unitlength}{\svgwidth}%
  \fi%
  \global\let\svgwidth\undefined%
  \global\let\svgscale\undefined%
  \makeatother%
  \begin{picture}(1,0.77632356)%
    \lineheight{1}%
    \setlength\tabcolsep{0pt}%
    \put(0,0){\includegraphics[width=\unitlength,page=1]{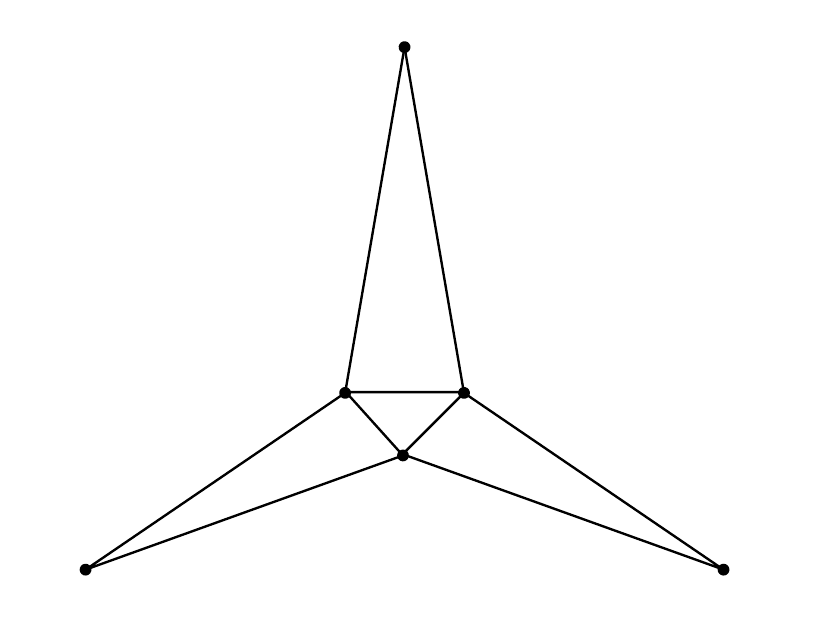}}%
    \put(0.36140976,0.514111){\makebox(0,0)[lt]{\lineheight{1.25}\smash{\begin{tabular}[t]{l}\textcolor{blue}{$n$}\end{tabular}}}}%
    \put(0.24097105,0.21585038){\makebox(0,0)[lt]{\lineheight{1.25}\smash{\begin{tabular}[t]{l}\textcolor{blue}{$n$}\end{tabular}}}}%
    \put(0,0){\includegraphics[width=\unitlength,page=2]{fat-slim-one.pdf}}%
    \put(0.3149947,0.11678929){\makebox(0,0)[lt]{\lineheight{1.25}\smash{\begin{tabular}[t]{l}\textcolor{blue}{$n$}\end{tabular}}}}%
    \put(0.61326656,0.11760575){\makebox(0,0)[lt]{\lineheight{1.25}\smash{\begin{tabular}[t]{l}\textcolor{blue}{$n$}\end{tabular}}}}%
    \put(0.69137242,0.22483125){\makebox(0,0)[lt]{\lineheight{1.25}\smash{\begin{tabular}[t]{l}\textcolor{blue}{$n$}\end{tabular}}}}%
    \put(0.56781825,0.51357802){\makebox(0,0)[lt]{\lineheight{1.25}\smash{\begin{tabular}[t]{l}\textcolor{blue}{$n$}\end{tabular}}}}%
  \end{picture}%
\endgroup%

    \caption{A geodesic triangle which is fat (as its slimness is $n/2$), but still has its insize 1}
    \label{fig:fat-slim-one}
\end{figure}

\section{Random models}

In the previous section, we defined the average notions of hyperbolicity and presented several examples that illustrate the \emph{discrepancy} amongst the average definitions. This discrepancy does not exist for the usual definitions of hyperbolicity. In this section, we show that those examples are, in fact, not special and are, in fact, fairly generic. In particular, we show that on a variety of random graphs the average definitions of hyperbolicity are not equivalent. The two families of random graphs we examine are random regular graphs and Erd\"os-Renyi random graphs.
\subsection{Random regular graph $\RRG(n,d)$}
In this section, we analyze the average hyperbolicity of the random regular graph model $\RRG(n,d)$ with $d \geq 3$. One might think that such graphs are \emph{not} hyperbolic on average as \cite{benjamini2011geodesics} showed that the random regular graph is \emph{not} hyperbolic (according to the usual definitions) and, in fact, the hyperbolicity constant is at least $\frac{1}{2} \log_{d-1}n - \omega(n)$ with high probability. This value is large considering that any graph is trivially $\frac{1}{2} \diam(G)$-hyperbolic and the diameter of a random regular graph is bounded by $\frac{1}{2} \log_{d-1}n + O(\log \log n)$). Furthermore, \cite{tucci2013non} investigated a technique to show these graphs are not hyperbolic using the observation that hyperbolic graphs must admit traffic congestion (or have a ``core''). 

We demonstrate that despite these results for the usual definitions of hyperbolicity on random regular graphs, there is an asymptotic gap between the average hyperbolicity and the average slimness. It is well-known that the random graph $\RRG(n,d)$ is connected with high probability (i.e., such event happens with probability $\to 1$), so that we condition on that.

First, we show that the average hyperbolicity is surprisingly bounded by $o(\omega(n))$. This is because the \emph{typical distance} between two nodes is reasonably concentrated, so that the generic four point condition is not a large enough multiple of the number of vertices. More precisely, we consider the \emph{distribution of shortest path lengths (DSPL)}, the distribution of distance between two fixed points. Our result builds on the recent analysis: 
\begin{lemm}\label{lemm:DSPL-RRG-bounds}
    \cite{tishby2022mean}~For fixed $x_1, x_2$, consider $L_n := d(x_1, x_2)$ in $G_n \in \RRG(n,d)$. Then the variance of $L_n$, $\Var(L_n)$, satisfies
    \begin{equation}
        \Var(L_n) = \frac{\pi^2}{6[\log(d-1)]^2} + \frac{1}{12} + 
             O \left( \frac{\log n}{n} \right).
        \label{eqn:DSPL}
    \end{equation}
\end{lemm}





\begin{thm}
    Given $G_n \in \RRG(n,d)$ for $d \geq 3$ and any growth function $\omega(n) \to \infty$,
    \[\mathbb{E}\left(\mathbb{E}\left(\delta_{fp}(G_n)\right)\right) = o(\omega(n)),\]
    with high probability.
\end{thm}

\begin{proof}
    First, we observe that because our random graph model is \emph{homogeneous}, for fixed $x_1, x_2, x_3, x_4 \in V(G_n)$,
    \[
        \mathbb{E}\left(\mathbb{E}(\delta_{fp}(G_n))\right) = 
            \mathbb{E}\fp(x_1, x_2, x_3, x_4)
    \]
    (note that we condition $G_n$ being connected). Let $P = d(x_1, x_2) + d(x_3, x_4)$ and $Q = d(x_1, x_3) + d(x_2, x_4)$. By the definition of the distribution of the distance between two fixed points (in a random regular graph), both $d(x_1,x_2)$ and $d(x_3,x_4)$ are distributed as $L_n$. Furthermore, even though $d(x_1, x_2)$ and $d(x_3, x_4)$ might not be independent, 
    \[
        \mathbb{E}[P] = \mathbb{E}[d(x_1, x_2) + d(x_3, x_4)] = 2 \mathbb{E}(L_n) \quad \text{ and } \quad \Var[P] \leq 4 \Var(L_n),
    \]
    and similarly for $\mathbb{E}[Q]$ and $\\Var[Q]$. Therefore,
    \begin{align*}
        \mathbb{E}\fp(x_1, x_2, x_3, x_4) &\leq \tfrac{1}{2} \mathbb{E}|P-Q| \leq \tfrac{1}{2} \left[ \mathbb{E}|P-2 \mathbb{E}(L_n)| +  \mathbb{E}|Q-2 \mathbb{E}(L_n)| \right] \\
        & \leq \tfrac{1}{2} \left[ \Var[P]^{1/2} + \Var[Q]^{1/2} \right] \leq 2 \Var(L_n)^{1/2}.
    \end{align*}
    Lemma~\ref{lemm:DSPL-RRG-bounds} tells us that the right hand side of the above inequality is bounded above by Equation~\ref{eqn:DSPL}. Thus, given any $\omega(n) \to \infty$,
    \[ \mathbb{P}\Big( \mathbb{E}(\delta_{fp}(G_n)) \geq \omega(n) \Big) \leq \frac{2 \Var(L_n)^{1/2}}{\omega(n)} \to 0.\]
\end{proof}

We emphasize that we barely used the structure of quadruple of points and the four point condition to obtain the previous result; the bound on the distance distribution itself already provides a bound on the average Gromov hyperbolicity! 

To prove the next theorem, we build upon the ideas from \cite{tucci2013non} and \cite{li2015traffic}to show that \emph{most} geodesic triangles (in a random regular graph) are, in fact, extremely fat!

\begin{thm}
    There exists a constant $C_d >0$, which only depends on $d$ (for $d \geq 3$), such that for $G_n \in \RRG(n,d)$,
    \[ \mathbb{E}\left(\zeta_{{\rm slim}}(G_n)\right), \mathbb{E}\left(\eta_{{\rm minsize}}(G_n)\right) \geq \frac{1}{2} \log_{d-1}n - \frac{3}{2}  \log_{d-1} \log_{d-1}n - C_d,\]
    with high probability.
    \label{thm:fat_triangle}
\end{thm}

\begin{coro}
    For any $\epsilon > 0$, with high probability, most geodesic triangles in $G_n$ are $\delta$-fat with 
    \[   
        \delta = \frac{1}{2} \log_{d-1}n - ( \frac{3}{2} + \epsilon) \log_{d-1} \log_{d-1}n.
    \]
\end{coro}

To prove Theorem~\ref{thm:fat_triangle}, we use the following sequence of Lemmas. Denote $\Gamma_r(w) := \{x \in V: d(x,w) = r\}$ and $N_r(w) := \{x \in V: d(x,w) \leq r\}$ for $w \in V$ and any integer $r \geq 0$. 
\begin{lemm}
    If $G = (V,E)$ is a graph with maximal degree $d \geq 3$, then
    \[ |\Gamma_r(x)| \leq d(d-1)^{r-1}, \quad |N_r(x)| < \frac{d}{d-2}(d-1)^r\] for all $x \in V, r \geq 0$.
\end{lemm}
\begin{proof} Given $x \in V$, consider a breadth first search tree on $G$ with root vertex $x$ for which $|\Gamma_r(x)|$ is exactly the number of nodes with level $r$. As every node except the root $x$ has at most $d-1$ children (and at most $d$ including $x$ itself), $|\Gamma_{r+1}(x)| \leq (d-1) |\Gamma_{r}(x)| \leq d(d-1)^{r}$ can be shown inductively. The second part of the Lemma can be shown by the calculation
    \begin{align*}
    |N_r(x)| = \sum_{i \leq r} |\Gamma_i(x)| & \leq 1 + d(1+(d-1)+\cdots+(d-1)^{r-1}) 
    \\ & = 1 + d \frac{(d-1)^r - 1}{d-2} < \frac{d}{d-2}(d-1)^r.
    \end{align*}
\end{proof}

In the next Lemma (from~\cite{tucci2013non}, we are interested in the following set of pairs of vertices, which is informally referred as a \emph{traffic} (and is also closely related to the \emph{betweenness centrality}), 
\[ T(w) := \{(x,y) \in V \times V: \text{ there exists a geodesic }[x,y] \text{ contains } w\} \quad \text{ for } w \in V.\]
\begin{lemm}\label{lemm:traffic-bounds}
    \cite{tucci2013non} Let $G = (V,E)$ be a graph with diameter $D$. For $w \in V$ and $i \geq 0$, we have
    \[ T(w) \leq \sum_{k+l \leq D} |\Gamma_k(w)| \cdot |\Gamma_l(w)| .\]
\end{lemm}
\begin{proof}
    Given $(x,y) \in T(w)$, it is clear that $d(x,w)+d(w,y) = d(x,y) \leq D$. The number of $(x,y) \in T(w)$ conditioned on $d(x,w)=k$ and $d(y,w)=l$ is at most $|\Gamma_k(w)| \cdot |\Gamma_l(w)|$. Therefore, the desired inequality holds.
\end{proof}

\begin{lemm}\label{lemm:T_bounds}
    Suppose $G_n \in \RRG(n,d)$ with $d \geq 3$ has its diameter $D$. Then
    \[ \left| T(w) \right| < \frac{d^2}{(d - 2)^2} (d-1)^D\]
    holds for all $w$ in $V(G_n)$.
\end{lemm}
\begin{proof}
    By lemma~\ref{lemm:traffic-bounds}, we see
    \begin{align*}
        |T(w)| & \leq \sum_{k+l \leq D} |\Gamma_k(w)| \cdot |\Gamma_l(w)| \\
        & \leq \sum_{k+l \leq D} d^2(d - 1)^{k+l-2} \\
        & = \frac{d^2}{(d - 1)^2} \sum_{0 \leq t \leq D} \sum_{k+l = t} (d - 1)^{t} \\
        & \leq \frac{d^2}{(d - 1)^2} \sum_{0 \leq t \leq D} (D+1-t) \cdot (d - 1)^{t} \\
        & = \frac{d^2}{(d - 1)^2} \left( \frac{(d-1)^{D+1} - 1}{d - 2} + \frac{(d-1)^{D} - 1}{d - 2} + \cdots + \frac{d - 2}{d - 2}\right) \\
        & < \frac{d^2}{(d - 1)^2} \left( \frac{(d-1)^{D+1} + (d-1)^D + \cdots + 1}{d - 2}\right) \\
        & = \frac{d^2}{(d - 1)^2} \cdot \frac{(d-1)^{D+2} - 1}{(d - 2)^2} < \frac{d^2}{(d - 2)^2} (d-1)^D.
    \end{align*}
    It turns out that the bound provided by \cite{tucci2013non} is quite loose and can be improved. This improvement will be used when we compute the second leading term in the bound on the sizes of the average geodesic triangle.
\end{proof}

In the next Lemmas, we use the known bound on the diameter for a random regular graph.
\begin{lemm}\label{lemm:RRG-diameter-bounds}
    \cite{bollobas1982diameter}~There exists a constant $C_d$, only depends on $d \geq 3$, such that $\diam(G_n) \leq \log_{d-1}n + \log_{d-1}\log_{d-1}n + C_d$ with high probability.
\end{lemm}

The next Lemma is an improvement upon that in \cite{li2015traffic}. 
\begin{lemm}\label{lemm:colors-count}
    Let $G = (U,V,E)$ be a bipartite graph. The edges of $G$ are colored in such a way that so that given a node $x \in U \cup V$, the number of colors incident to $x$ is at most $t(x)$. ($u$ is incident to a color $c$ if $u$ is incident to an edge with color $c$). Then there exists a color that is used by at least
    \[ \frac{|E|^2}{\left( \sum_{u \in U} t(u) \right) \left( \sum_{v \in V} t(v) \right)} \]
    edges in $E$.
\end{lemm}
Note that the above lemma immediately implies Lemma 3.2 in \cite{li2015traffic} (by letting $|U| = |V| = n$ and $t(x) \leq t$ for all nodes $x$). The reason we developed a new lemma is as follows. Instead of the universal bound on $t$ (which is related to the behavior of the diameter of random graphs), we are able to use the \emph{average} bound on $t$, which may be used to improve the bound later.
\begin{proof}
    We provide the details of the proof in the Appendix.
\end{proof}

Now we are ready to prove the main theorem.
\begin{proof}
    First, we will prove the statement on the \emph{minsize}. We begin with the \emph{sampling procedure} below.
    \begin{itemize}
        \item For every pair $(v,w)$, \emph{fix} a geodesic segment between $v$ and $w$ (by randomly choosing one). For the rest of the argument, we use only these segments when we consider a geodesic triangle $\Delta(x,y,z)$. These choices constitute a path system.
        \item Randomly choose a pivot point $z \in V(G_n)$. Then evaluate $\minsize(\Delta(x,y,z))$ for any $x,y \in V(G_n)$.
    \end{itemize}
    It is easy to see that
    \[ 
        \mathbb{E} \sum_{x,y \in V} \minsize(\Delta(x,y,z)) = n^2 \mathbb{E} \eta_{{\rm minsize}}(G_n).
    \]
    Consider a threshold $\delta \geq 0$ and construct the following bipartite graph $H = (X,Y,E)$ with $X = Y = V(G_n)$ and 
    \[ 
        (x,y) \in E \Leftrightarrow \minsize(\Delta(x,y,z)) \leq \delta .
    \]
    We color each edge in $(x,y) \in E$ with $c(x,y) = p_z \in [x,y]$, where $\{p_x, p_y, p_z\}$ realizes the minsize of geodesic triangle. Then by assumption, $d(p_z, [x,z]), d(p_y, [y,z]) \leq \delta$. We see that
    \begin{itemize}
        \item In any $x \in X$, the number of colors incident to $x$ is at most $|N_{\delta}([x,z])| \leq (d(x,z) + 1) \cdot N(\delta)$, where $N(\delta)$ is the universal bound on the size of $\delta$-neighborhood. This is because for $p_z \in [x,y]$, there exists $p_x \in [x,z]$ such that $p_z \in N_{\delta}(p_x) \subset N_{\delta}([x,z])$. 
        \item Given any color $c$, the number of edges colored with $c$ is at most $|T(c)|$ where $T(c) = \{(x,y) \in V \times V: c \in [x,y]\}$, which is bounded by lemma.
    \end{itemize}
    By Lemma~\ref{lemm:colors-count}, we have
    \[ \frac{d^2}{(d - 2)^2} (d-1)^D \geq |T(c)| \geq \max_{c \in \mathcal{C}} |E_c| \geq \frac{|E|^2}{(\sum_{x \in V} |N_{\delta}(x,z)|)^2} \geq \frac{|E|^2}{N(\delta)^2 (n + \sum_{x \in V} d(x,z))^2} .\]
    We have $N(\delta) \leq \frac{d}{d-2}(d-1)^{\delta}$ and $\sum_{x \in V} d(x,z) \leq n \diam(G_n)$. This shows that given $p := |E| / n^2$,
    

    \[ p \leq \frac{1}{n}\left[ \frac{d^2}{(d-2)^2} (\diam(G_n)+1)(d-1)^{t + \frac{\diam(G_n)}{2}} \right] .\]
    We note that $p$ is the probability that our \emph{random} geodesic triangle has the slimness at most $t$. Therefore, if we take expectation, we have
    \[ \mathbb{P}(\minsize(\Delta(x,y,z)) \leq t) \leq \frac{1}{n}\left[  \frac{d^2}{(d-2)^2} (\diam(G_n)+1)(d-1)^{t + \frac{\diam(G_n)}{2}} \right] \]
    \[ \mathbb{P}(\minsize(\Delta(x,y,z)) > t) \geq 1 - \frac{1}{n}\left[  \frac{d^2}{(d-2)^2} (\diam(G_n)+1)(d-1)^{t + \frac{\diam(G_n)}{2}} \right] .\]
    Since $\diam(G_n) \leq \log_{d-1}n + \log_{d-1}\log_{d-1}n + C_d$ with high probability, this suggests that w.h.p.,
    \[ \mathbb{P}(\minsize(\Delta(x,y,z)) > t) \geq 1 - \frac{1}{n}\left[ \frac{d^2 (d-1)^{C_d/2}}{(d-2)^2} (\log_{d-1}n + \log_{d-1}\log_{d-1}n + (C_d + 1)) \cdot \sqrt{n \cdot \log_{d-1}n} \cdot (d-1)^t \right] .\]
    Denote $C_1 := C_d + 1$ and $C_2 := \frac{d^2 (d-1)^{C_d/2}}{(d-2)^2 \log(d-1)}$.
    Given $S > 0$ (which will be chosen later), we have
    \begin{align*}
        \mathbb{E} \zeta_{minsize}(G_n) & = \int_{0}^{\infty} \mathbb{P}(\minsize(\Delta(x,y,z)) > t) dt \geq \int_{0}^{S} \mathbb{P}(\minsize(\Delta(x,y,z)) > t) dt \\
        & = S - \int_0^S \frac{1}{n}\left[ \log(d-1) C_2 (\log_{d-1}n + \log_{d-1}\log_{d-1}n + C_1) \cdot \sqrt{n \cdot \log_{d-1}n} \cdot (d-1)^t \right] \\
        & = S - \log(d-1) C_2 \frac{(\log_{d-1}n + \log_{d-1}\log_{d-1}n + C_1) (\log_{d-1}n)^{1/2}}{n^{1/2}} \int_0^S (d-1)^t dt \\
        & = S - \log(d-1) C_2 \frac{(\log_{d-1}n + \log_{d-1}\log_{d-1}n + C_1) (\log_{d-1}n)^{1/2}}{n^{1/2}} \left[\frac{(d-1)^S - 1}{\log(d-1)} \right] \\
        & > S - C_2 \frac{(\log_{d-1}n + \log_{d-1}\log_{d-1}n + C_1) (\log_{d-1}n)^{1/2}}{n^{1/2}} \cdot (d-1)^S.
    \end{align*}
    Plugging in $S = \frac{1}{2} \log_{d-1}n - \frac{1}{2} \log_{d-1} \log_{d-1}n - \log_{d-1}( \log_{d-1} n + \log_{d-1} \log_{d-1} n + C_1 )$ yields
    \[\mathbb{E} \zeta_{minsize}(G_n) > S - C_2,\]
    as desired. The theorem can be shown by observing that $|S - \frac{1}{2} \log_{d-1}n + \frac{3}{2}  \log_{d-1} \log_{d-1}n| = o(1)$.

    For the slimness case, we need to choose a different value for $c$, given a \emph{slim} triangle $\Delta(x,y,z)$. In this case, we choose $c \in [x,y]$ based on the following simple proposition.
    \begin{lemm}
        If $\slim(\Delta(x,y,z)) = \delta$, then there exists a vertex $c \in [x,y]$ such that $d(c, [x,z]), d(c, [y,z]) \leq \delta + 1$.
    \end{lemm}
    \begin{proof}
        Pick $c \in [x,y]$ with $d(c, [x,z]) > \delta$, which is the \emph{closest} from $x$. If such $c$ does not exist, then the case is obvious. By the slimness, $d(c,[y,z]) \leq \delta$ holds. Also, for $c'$ the adjacent vertex of $c$ in $[x,y]$ toward $x$, $d(c,[x,z]) \leq d(c',[x,z]) + 1 \leq \delta + 1$ holds.
    \end{proof}
    Once we have done that, we can repeat the exactly same argument up to an additional constant 1.
\end{proof}

\begin{remark}
    We improved lemmas used in \cite{tucci2013non} in order to achieve the second leading term $-\frac{3}{2}  \log_{d-1} \log_{d-1}n$. If we had directly used the original lemmas, we would obtain the bound $- \frac{5}{2}  \log_{d-1} \log_{d-1}n$ instead.
\end{remark}

Intuitively, this suggests the following. It is easy to check that the slimness of any geodesic triangle is bounded by $\frac{1}{2} \diam(G)$, which is about $\frac{1}{2}\log_{d-1}n + \frac{1}{2}\log_{d-1}\log_{d-1}n$. On the other hand, we have seen that \emph{most} geodesic triangles in a random regular graph are fat so that the slimness is at least about $\frac{1}{2}\log_{d-1}n - \frac{3}{2}\log_{d-1}\log_{d-1}n$. Therefore the average behavior is somewhere between $[\frac{1}{2}\log_{d-1}n - \frac{3}{2}\log_{d-1}\log_{d-1}n , \frac{1}{2}\log_{d-1}n + \frac{1}{2}\log_{d-1}\log_{d-1}n]$. 


\subsection{Erd\H{o}s-R\'enyi graph $\ER(n,\lambda / n)$ with $\lambda > 1$}

Finally, in this subsection, we consider the Erd\H{o}s-R\'enyi random graph model $\ER(n,p)$. In order to ensure that the graphs we study have a large connected component, we consider the supercritical case in which the probability of an edge between any two of the $n$ vertices is $p = \lambda / n$ for some $\lambda > 1$. Following convention, we consider the giant component of these graphs, which has size roughly $\gamma(\lambda) \cdot n$. \cite{narayan2015lack} showed that this sparse random graph is \emph{not} hyperbolic by constructing a \emph{single} geodesic triangle which is forced to be fat. (In fact, their result suggests that such triangle is $\Omega(\log n)$-fat, which is not explicitly stated in the original paper.) We want something more in this average case!

Our main result is as follows.
\begin{thm}\label{thm:ERtriangleisfat}
    Given $G_n \in \ER(n,\lambda / n)$ with $\lambda \geq 4.67$, there exists $f(\lambda) > 0$ such that with high probability,
    \[ \mathbb{E}\left(\zeta_{slim}(\mathscr{C}_{\max})\right), \mathbb{E}\left(\eta_{{\rm minsize}}(\mathscr{C}_{\max})\right) \geq f(\lambda) \cdot \log n,\]
    where $\mathscr{C}_{\max}$ denotes the giant component of $G_n$.
\end{thm}

\begin{coro}
    With high probability, a geodesic triangle in (the giant component $\mathscr{C}_{\max}$ of) $G_n \in \ER(n, \lambda / n)$ is $\Omega(\log n)$-fat.
\end{coro}

The corollary answers the question raised in \cite{narayan2015lack} and shows that the $\Omega(\log n)$-fat triangle in the Erd\H{o}s-R\'enyi graph is \emph{typical}.

\begin{proof}
    The proof procedure is as follows. We use a similar method to that for the random regular graph case. While those calculations were detailed, here, we omit the details and argue simply that the leading term is $O(\log n)$.
    \begin{enumerate}
        \item First, given a point $x$, we want to bound the number of geodesics which contains the point $x$. Following the arguments for Lemmas~\ref{lemm:traffic-bounds} and \ref{lemm:T_bounds}, what we need to do is to bound the size of neighborhood $|\Gamma_r(x)|$ and $|N_r(x)|$ given $r > 0$. However, as the degree is not bounded, so that the upper bound is not immediate. it is not clear. Hence we adapted the lemma stated from \cite{chung2001diameter}.
        \item We also need to bound the diameter of our giant component. We use explicite results from \cite{riordan2010diameter}. The condition $\lambda \geq 4.67$ is required so that the diameter $\diam(\mathscr{C}_{max})$ is not bigger than $2 \log n / \log \lambda + O(1)$.
    \end{enumerate}

\begin{lemm}[Bounds on the size of neighborhood]\label{lemm:neighbor-bounds-on-ER}
    Given $G_n \in \ER(n,\lambda / n)$ with $\lambda > 1$,
    \[ |\Gamma_r(x)| \leq (r+1)^2 \lambda^r \cdot \log n\]
    holds for all $x \in V(G_n)$ and $0 \leq r \leq n$, with high probability.
\end{lemm}

\begin{proof}
    We provide the details of the proof in the Appendix for completeness. 
\end{proof}

\begin{lemm}[Bounds on the diameter, \cite{riordan2010diameter}]
    Given $\lambda > 1$, let $\lambda_* < 1$ satisfy $\lambda_* e^{-\lambda_*} = \lambda e^{-\lambda}$ (i.e., the conjugate). Then there exists a constant $C_\lambda$ that depends only on $\lambda$, such that for $G_n \in \ER(n, \lambda/n)$ and its giant component $\mathscr{C}_{max}$,
    \[ \diam(\mathscr{C}_{max}) \leq  \left(\frac{1}{\log \lambda} + \frac{2}{\log (1/\lambda_*)} \right) \log n + C_\lambda \]
    with high probability. Note that for $\lambda \geq 4.67, \frac{2}{\log (1/\lambda_*)} < \frac{1}{\log \lambda}$ holds.
\end{lemm}

We are now ready to prove Theorem~\ref{thm:ERtriangleisfat}. Given $\lambda \geq 4.67$, pick $0 < \epsilon < 1 - \frac{2 \log \lambda}{\log(1/\lambda_*)}$. With high probability, $\diam(\mathscr{C}_{max}) \leq (2-\epsilon) \frac{\log n}{\log \lambda} + C_\lambda (=:D_\lambda)$ so that by Lemmas~\ref{lemm:traffic-bounds} and \ref{lemm:neighbor-bounds-on-ER},
\begin{align*}
    |T(w)| & \leq \sum_{k+l \leq D_\lambda} |\Gamma_k(w)| \cdot |\Gamma_l(w)| \\
    & \leq  \sum_{k+l \leq D_\lambda} (k+1)^2(l+1)^2 \lambda^{k+l} (\log n)^2 \\
    & \leq (\log n)^2 \sum_{t \leq D_\lambda} \sum_{k+l = t} (k+1)^2 (l+1)^2 \lambda^{t} \\
    & = \frac{(\log n)^2}{30} \sum_{t \leq D_\lambda} (t-1)(t+1)(t+2)(t+3)(t+5)\lambda^t \\
    & \leq \frac{(\log n)^2 (D_\lambda+5)^5}{30}  \sum_{t \leq D_\lambda} \lambda^t < \frac{(\log n)^2 (D_\lambda+5)^5 \lambda^{D_\lambda+1}}{\lambda - 1},
\end{align*}
holds for all $w \in \mathscr{C}_{max}$. We see that, after substituting in the expression for $D_
lambda$ and simplifying, the right hand side becomes $O(n^{2-\epsilon} \log^7 n)$. In other words, there exists $A_\lambda > 0$ such that
\[ |T(w)| \leq A_\lambda \cdot n^{2-\epsilon} \log^7 n,\]
for $n$ sufficiently large.

Next, we use our previous proof technique again. This gives us that for $p:= \mathbb{P}(\minsize(\Delta(x,y,z)) \leq t)$, for a giant component $\mathscr{C}_{max}$ with size $\gamma n$, there exists a color $c$ which is used at least
\[ \frac{(p(\gamma n)^2)^2}{N(t)^2 (\gamma n + \sum_{x \in \mathscr{C}_{max}} d(x,z))^2} \leq |T(c)| \leq A_\lambda \cdot n^{2-\epsilon} \log^7 n \]
times.

We simply use $d(x,z) \leq \diam(\mathscr{C}_{max}) \leq 2 \frac{\log n}{\log \lambda} - 1$ and $N(t) \leq \sum_{r \leq t} \max |\Gamma_r(x)| \leq \sum_{r \leq t} (r+1)^2 \lambda^r \log n \leq (t+1)^3 \lambda^t \log n$, so that

\[ p = \mathbb{P}(\minsize(\Delta(x,y,z)) \leq t) \leq \frac{2\sqrt{A_\lambda}}{\gamma \log \lambda} n^{-\epsilon/2} (\log^{5.5}n) \cdot (t+1)^3 \lambda^t, \]
\[ 1 - p = \mathbb{P}(\minsize(\Delta(x,y,z)) > t) \geq 1 - \frac{2\sqrt{A_\lambda}}{\gamma \log \lambda} n^{-\epsilon/2} (\log^{5.5}n) \cdot (t+1)^3 \lambda^t. \]
Now pick $S_n > 0$ such that $\int_0^{S_n} (t+1)^3 \lambda^t dt= \left( \left[ \frac{\lambda^t((t+1)^3 \log^3 \lambda - 3(t+1)^2 \log^2 \lambda + 6(t+1) \log \lambda - 6)}{\log^4 \lambda} \right]^{t=S_n}_{t=0} = \right) \frac{n^{\epsilon/2}}{\log^{5.5}n}$. It turns out that
\[ \mathbb{E} \zeta_{minsize}(\mathscr{C}_{max}) = \int_{0}^{\infty} \mathbb{P}(\minsize(\Delta(x,y,z)) > t) dt \geq \int_{0}^{S_n} \mathbb{P}(\minsize(\Delta(x,y,z)) > t) dt \geq S_n - \frac{2 \sqrt{A_\lambda}}{\gamma \log \lambda}.\]
It is easy to check that $S_n \sim (\epsilon/2) \cdot \frac{\log n}{\log \lambda}$. Also, with high probability, $\gamma$ converges to the final probability $\gamma(\lambda)$, which only depends on $\lambda$. Therefore, the right hand side of the inequality is bounded by $\Omega(\log n)$ with high probability, thus completing the proof.

The slimness case can be shown analogously.
\end{proof}

As we have seen in random regular graph case, we again expect a gap between the (four point) hyperbolicity and the slimness of the geodesic triangles. Although, we do not have a proof explicitly bounding the four point condition, we instead include a simple proof that the gap \emph{exists}.

\begin{thm}
Given $G_n \in \ER(n, \lambda/n)$ with $\lambda > 1$ and its giant component $\mathscr{C}_{\max}$, we have
    \[ \mathbb{E}\left(\delta_{fp}(\mathscr{C}_{max})\right) = o(\log n)\]
with high probability.
\end{thm}

\begin{proof}
    Fix $S_{\lambda} > \frac{1}{\log \lambda} + \frac{2}{\log (1/\lambda_*)}$.
    Let $B$ be the event such that \begin{enumerate}
        \item The giant component $\mathscr{C}_{max}$ has size $0.99 \gamma(\lambda) \cdot n \leq |\mathscr{C}_{max}| \leq 1.01 \gamma(\lambda) \cdot n$.
        \item The diameter $\diam(\mathscr{C}_{max}) \leq S_{\lambda} \cdot \log n$.
    \end{enumerate}
    We see that $\mathbb{P}(B) \to 1$ as $n \to \infty$. 
    Let $B_{1234}$ be the event that $B$ occurs and $x_1, x_2, x_3, x_4$ all belong to the giant component $\mathscr{C}_{\max}$ (for fixed vertices $x_1, x_2, x_3, x_4$). As the graph model is homogeneous, $\mathbb{P}(B_{1234}) \leq \frac{(1.01\gamma(\lambda)n)^4}{n^4}\mathbb{P}(B) = (1.01\gamma(\lambda))^4\mathbb{P}(B)$.

    The expectation of average hyperbolicity \emph{conditioned on} $B$ is
    \begin{align*}
        \mathbb{E}\left(\mathbb{E}\delta_{fp}(\mathscr{C}_{\max}) \;|\; B \right) & = \mathbb{E} \left( \frac{1}{|\mathscr{C}_{max}|^4} \sum_{x,y,z,w \in \mathscr{C}_{\max}} \fp(x,y,z,w) \;|\; B \right) \\
        & \leq \frac{1}{(0.99 \gamma(\lambda) n)^4} \mathbb{E} \left( \sum_{x,y,z,w \in \mathscr{C}_{\max}} \fp(x,y,z,w) \;|\; B \right) \\
        & = \frac{1}{(0.99 \gamma(\lambda) n)^4} \sum_{x,y,z,w} \mathbb{E} \left( \mathds{1}(x,y,z,w \in \mathscr{C}_{\max}) \wedge \fp(x,y,z,w) \;|\; B \right) \\
        & = \frac{1}{(0.99 \gamma(\lambda) n)^4} \frac{n!}{(n-4)!} \mathbb{E} \left( \mathds{1}(x_1,x_2,x_3,x_4 \in \mathscr{C}_{\max}) \wedge \fp(x_1,x_2,x_3,x_4) \;|\; B \right) \\
        & \leq \frac{1}{(0.99\gamma(\lambda))^4} \frac{\mathbb{P}(B_{1234})}{\mathbb{P}(B)} \mathbb{E} \left( \fp(x_1,x_2,x_3,x_4) \;|\; B_{1234} \right) \\
        & \leq \frac{(1.01)^4}{(0.99)^4} \mathbb{E} \left( \fp(x_1,x_2,x_3,x_4) \;|\; B_{1234} \right).
    \end{align*}
    Now we use the result from \cite{van2021giant}. It states that for fixed $x_1, x_2$, $\frac{d(x_1, x_2)}{\log n}$ converges to $\frac{1}{\log \lambda}$ in probability, while conditioned on $x_1$ and $x_2$ belonging to the giant component $\mathscr{C}_{max}$. In other words, given any $\eta > 0$,
    \[ \mathbb{P}(|d(x_1, x_2) - \log_{\lambda}n| > \eta \log_{\lambda}n \; | \; x_1, x_2 \in \mathscr{C}_{max}) \to 0. \]

    Fix $\eta > 0$. Let $E$ be the event (conditional on $B_{1234}$) such that there exists a pairwise distance among $x_1, x_2, x_3, x_4$ outside of $[(1-\eta)\log_{\lambda}n, (1+\eta)\log_{\lambda}n]$. If $E$ does not occur, then the four point condition $\fp(x_1, x_2, x_3, x_4)$ is bounded by $\eta \log_{\lambda}n$. If $E$ does occur, then we use $\fp(x_1, x_2, x_3, x_4) \leq \frac{1}{2} \diam(\mathscr{C}_{max}) \leq \frac{1}{2} S_{\lambda} \log n$ so that with high probability,
    \begin{align*}
        \mathbb{E}\fp(x_1, x_2, x_3, x_4 \;|\; B_{1234}) & = \mathbb{E}\left(\fp(x_1, x_2, x_3, x_4) \;|\; E \right) \mathbb{P}(E \;|\; B_{1234}) + \mathbb{E}\left(\fp(x_1, x_2, x_3, x_4) \;|\; E^c \right) \mathbb{P}(E^c \;|\; B_{1234})\\
        & \leq \frac{S_{\lambda}}{2}  \log n \cdot \mathbb{P}(E \;|\; B_{1234}) + (\eta \log_{\lambda}n) \cdot (1 - \mathbb{P}(E \;|\; B_{1234})) \\
        & \leq \left( \eta + \left(\frac{S_{\lambda}}{2} - \eta\right) \mathbb{P}(E \;|\; B_{1234}) \right) \cdot \log_{\lambda}n.
    \end{align*}
    As $\mathbb{P}(E \;|\; B_{1234}) \to 0$ as $n \to \infty$, $\mathbb{E}\fp(x_1, x_2, x_3, x_4 \;|\; B_{1234}) < 1.01 \eta \log_{\lambda}n$ for $n$ sufficiently large.

    Hence, for any fixed $\eta > 0$ and $M > 0$, $\mathbb{E}\left(\mathbb{E}\delta_{fp}(\mathscr{C}_{max}) \;|\; B \right) \leq \frac{(1.01)^5}{(0.99)^4} \frac{\eta}{2M} \log_{\lambda}n < \frac{\eta}{M} \log_{\lambda}n$ holds for $n$ sufficiently large. Hence,
    \[ \mathbb{P}(\mathbb{E}\delta_{fp}(\mathscr{C}_{max}) \geq \eta \log_{\lambda}n \;|\; B) \leq \frac{\mathbb{E}\left(\mathbb{E}\delta_{fp}(\mathscr{C}_{max}) \;|\; B \right)}{\eta \log_{\lambda}n} < \frac{1}{M} \quad \text{ for } n \text{ sufficiently large.}\]
    We conclude that $\mathbb{E}\delta_{fp}(\mathscr{C}_{\max}) = o(\log n)$ with high probability (using the fact that $B$ occurs with high probability).
\end{proof}

\section{Discussion}

Our analysis establishes important distinctions between the usual and the average case definitions of hyperbolicity. We show that in contrast to the usual definitions, the average definitions are {\emph not} all equivalent up to a constant factor. Furthermore, the examples that exhibit these differences are themselves interesting cases of random discrete structures and distributions of points in high dimensional Euclidean space that might masquerade as points from a hyperbolic space if one is not astute in one's measurements! There are a number of open questions that arise as a result of our analysis.

\subsection{Improving the gap between AvgHyp and AvgSlim (or AvgMinsize)}

For the Erd\"os-Renyi random graph $\ER(n, \lambda/n)$ case, we establish that the average hyperbolicity (AvgHyp) is bounded by $o(\log n)$, using the distribution of the shortest path lengths (DSPL). Importantly, this result does not rely on any specific structural information about the quadruples of vertices, leaving room to address the gaps amongst the average case definitions. Similarly, in the case of random regular graphs, one can compute the variance of DSPL, which  provides a bound on the average hyperbolicity, without using specific structural information about quadruples of points.

We did, however, assume $\lambda \geq 4.67$ to ensure that the graph’s diameter remains manageable. What happens if $\lambda > 1$ but is much smaller, perhaps approaching 1? In such scenarios, can we still expect the graph and its geodesic triangles to be $\Omega(\log n)$-fat? Or does it become \emph{truly} tree-like instead? These are questions that are valid, interesting future or follow-up work after our initial results.

\subsection{``tree-likeness'' of AvgSlim (or AvgMinsize)}

We have seen that the \emph{average hyperbolicity} can be used as a measure of how tree-like the space is \cite{chatterjee2021average,yim2024fitting}. We have demonstrated that, actually, the average hyperbolicity is the weakest one amongst the average measures! In fact, the upper bound provided by \cite{chatterjee2021average} and \cite{yim2024fitting} are both quite loose: the quantitative bound on results of \cite{chatterjee2021average} is hard to track since they use the weighted Szemer\'edi regularity lemma, which is known to have a gigantic bound. Also, it turns out that the $O(n^3)$ bound given by \cite{yim2024fitting} is enormous in practice. This might not be a coincidence, as we may have assumed too weak an average hyperbolicity bound. For example, given a space or a graph where the \emph{average slimness} is bounded, is there a good tree that fits the given distance with smaller average distortion?


\subsection{Expander graphs}
This is an important open question. There are negative results for the usual definitions of hyperbolicity~\cite{benjamini1998expanders, malyshev2015expanders} but we do not know about the average definition. It's possible (given the results in~\cite{li2015traffic}) that there are some families of expanders that are hyperbolic on average and possibly some that are not.

\bibliography{ref}
\bibliographystyle{plain}

\newpage
\section{Appendix}

\subsection{Proof of Lemma~\ref{lemm:colors-count}}

    We will first show the following lemma:
    \begin{lemm}
        Given a bipartite graph $G = (U,V,E)$,
        \[ \sqrt{|E|} \leq \sum_{(u,v) \in E, u \in U, v \in V} \frac{1}{\sqrt{\deg(u) \cdot \deg(v)}}.\]
    \end{lemm}
    \begin{proof}
        \begin{align*}
            \sum_{(u,v) \in E, u \in U, v \in V} \frac{1}{\sqrt{\deg(u) \cdot \deg(v)}} & = \sum_{u \in U} \frac{1}{\sqrt{\deg(u)}} \left( \sum_{v \in N(u)} \frac{1}{\sqrt{\deg(v)}} \right) \\
            & \geq \sum_{u \in U} \frac{1}{\sqrt{\deg(u)}} \left( |N(u)| \cdot \frac{1}{\sqrt{M_u / |N(u)|}} \right) \quad (M_u := \sum_{v \in N(u)} \deg(v) \leq |E|) \\
            & = \sum_{u \in U} \frac{1}{\sqrt{\deg(u)}} \cdot \frac{\deg(u)^{3/2}}{\sqrt{M_u}} = \sum_{u \in U} \frac{\deg(u)}{\sqrt{M_u}} \geq \sum_{u \in U} \frac{\deg(u)}{\sqrt{|E|}} = \sqrt{|E|}.
        \end{align*}
    \end{proof}
    Now we will proceed the main proof of the Lemma. Let $\mathcal{C}$ denote the collection of colors and $c(x,y)$ denote the color of $(x,y) \in E$. Let $E_c := \{(x,y) \in E: c(x,y) = c\}$, $N_c(x) := \{y : (x,y) \in E_c\}$, and $\deg_c(x) := |N_c(x)|$ for $c \in \mathcal{C}$ and $x \in U \cup V$. Then by Cauchy-Schwarz inequality,
    \begin{align*}
        \left( \sum_{\substack{(u,v) \in E, u \in U, v \in V \\ c = c(u,v)}} \frac{1}{\sqrt{\deg_c(u) \cdot \deg_c(v)}} \right)^2 \leq \left( \sum_{\substack{(u,v) \in E, u \in U, v \in V \\ c = c(u,v)}} \frac{1}{\deg_c(u)} \right) \left( \sum_{\substack{(u,v) \in E, u \in U, v \in V \\ c = c(u,v)}} \frac{1}{\deg_c(v)} \right).
    \end{align*}
    We see that
    \[ \sum_{\substack{(u,v) \in E, u \in U, v \in V \\ c = c(u,v)}} \frac{1}{\deg_c(u)} = \sum_{u \in U} \sum_{v \in N(u), c = c(u,v)} \frac{1}{\deg_c(u)} = \sum_{u \in U} \sum_{\substack{c \in \mathcal{C} \\ N_c(u) \neq \emptyset}} \sum_{v \in N_c(u)} \frac{1}{\deg_c(u)} = \sum_{u \in U} \sum_{\substack{c \in \mathcal{C} \\ N_c(u) \neq \emptyset}} 1 \leq \sum_{u \in U} t(u),\]
    and similarly,
    \[ \sum_{\substack{(u,v) \in E, u \in U, v \in V \\ c = c(u,v)}} \frac{1}{\deg_c(v)} \leq \sum_{v \in V} t(v) \]
    holds. On the other hand,
    \[ \sum_{\substack{(u,v) \in E, u \in U, v \in V \\ c = c(u,v)}} \frac{1}{\sqrt{\deg_c(u) \cdot \deg_c(v)}} = \sum_{c \in \mathcal{C}} \sum_{(u,v) \in E_c, u \in U, v \in V} \frac{1}{\sqrt{\deg_c(u) \cdot \deg_c(v)}} \geq \sum_{c \in \mathcal{C}} \sqrt{|E_c|},\]
    by the previous lemma. This shows
    \[ \sum_{\substack{(u,v) \in E, u \in U, v \in V \\ c = c(u,v)}} \frac{1}{\sqrt{|E_c|}} = \sum_{c \in \mathcal{C}} \sqrt{|E_c|} \leq \sqrt{ \left( \sum_{u \in U} t(u) \right) \cdot \left( \sum_{v \in V} t(v) \right)}, \]
    so that there exists $(u,v) \in E$ with $\frac{1}{\sqrt{|E_c|}} \leq \frac{\sqrt{ \left( \sum_{u \in U} t(u) \right) \cdot \left( \sum_{v \in V} t(v) \right)}}{|E|}$ holds. Thus,
    \[ |E_c| \geq \frac{|E|^2}{\left( \sum_{u \in U} t(u) \right) \left( \sum_{v \in V} t(v) \right)}, \]
    as desired.






\subsection{Proof of Lemma~\ref{lemm:neighbor-bounds-on-ER}}
For fixed $x$, denote $a_t := |\Gamma_t(x)|$ and $b_t := \lfloor (t+1)^2 \lambda^t \log n \rfloor$. We will show the following:
\begin{claim}
    Given $0 \leq t \leq n$, $a_s \leq b_s$ holds for all $0 \leq s \leq t$ with probability at least $1 - t/n^{2\lambda}$.
\end{claim}
\begin{claimproof}
    We induct on $t$. For $t = 0$, $1 = a_0 \leq b_0 = \lfloor \log n \rfloor$, which is always true (if $n \geq 3$). Now, suppose the assertion holds on $t = r$. 
    The following Chernoff type bound is well-known.
    \begin{lemm}
        Given a binomial distribution $X \sim \Bin(n,p)$, we have
        \[ \mathbb{P}(X \geq np + a) \leq \exp\left( - \frac{a^2}{2(np+a/3)} \right) \quad \forall t \geq 0. \]
    \end{lemm}
    We use the fact that \emph{conditional on $a_0, a_1, \cdots, a_{r}$}, the distribution of $a_{r + 1}$ is dominated by the distribution of the number of boundary edges $|\partial \Gamma_{r}(x)|$, which is binomial with parameters $(n - |N_{r}(x)|) \cdot a_{r} \leq na_r$ and $p = \lambda / n$. Suppose $a_0 \leq b_0, \cdots, a_{r} \leq b_{r}$. Then
    \begin{align*}
        \mathbb{P}(a_{r+1} > b_{r+1} | a_0, \cdots, a_r) & = \mathbb{P}(a_{r+1} >(r+2)^2 \lambda^{r+1} \log n | a_0, \cdots, a_r) \\
        & \leq \mathbb{P}\left(\Bin \left(na_r, \frac{\lambda}{n} \right) > (r+2)^2 \lambda^{r+1} \log n \right) \\
        & \leq \mathbb{P}\left(\Bin \left(nb_r, \frac{\lambda}{n} \right) > (r+2)^2 \lambda^{r+1} \log n \right) \quad (\because a_r \leq b_r) \\
        & \leq \exp \left[ - \frac{\left( (r+2)^2 \lambda^{r+1} \log n - \lambda b_r \right)^2}{2 \left( \lambda b_r + \frac{(r+2)^2 \lambda^{r+1} \log n - \lambda b_r}{3}  \right)}\right] \\
        & \leq \exp \left[ - \frac{\left( (r+2)^2 \lambda^{r+1} \log n - (r+1)^2 \lambda^{r+1} \log n \right)^2}{2 \left( \frac{2 (r+1)^2 \lambda^{r+1} \log n + (r+2)^2 \lambda^{r+1} \log n}{3}  \right)}\right] \quad (\because \lambda b_r \leq (r+1)^2 \lambda^{r+1} \log n ) \\
        & = \exp \left[ - \lambda^{r+1} \log n \cdot \frac{3(4r^2 + 12r + 9)}{2(3r^2 + 8r + 6)} \right] \leq \exp \left( - 2 \lambda^{r+1} \log n \right) \leq \frac{1}{n^{2\lambda}}.
    \end{align*}
    By the induction hypothesis, we see that $a_0 \leq b_0, \cdots, a_r \leq b_r$ holds with probability at least $1 - r/n^{2\lambda}$ and conditional on that, the probability of $a_{r+1} > b_{r+1}$ is bounded by $1/n^{2\lambda}$. Therefore, $a_0 \leq b_0, \cdots, a_{r+1} \leq b_{r+1}$ holds with probability at least $1 - (r+1)/n^{2\lambda}$ which concludes the proof.
\end{claimproof}

Now, let $B(x)$ be the \emph{bad} event that $|\Gamma_t(x)| > (t+1)^2 \lambda^t \log n$ for some $t$. By the claim, $\mathbb{P}(B(x)) \leq n^{1-2\lambda}$. Therefore, for $B := \cup_x B(x)$,
\[ \mathbb{P}(B) \leq \sum_x \mathbb{P}(B(x)) \leq n^{2 - 2\lambda} = o(1),\]
as desired.

\end{document}